\documentclass[reqno]{amsart}
\usepackage{CJK}
\usepackage{hyperref}
\usepackage{cite}
\usepackage{amsmath}
\usepackage{mathrsfs}
\def\dive{\operatorname{div}}

\numberwithin{equation}{section}

\newtheorem{theorem}{Theorem}[section]
\newtheorem{lemma}[theorem]{Lemma}
\newtheorem{definition}[theorem]{Definition}

\newtheorem{proposition}[theorem]{Proposition}
\newtheorem{remark}[theorem]{Remark}

\allowdisplaybreaks

\makeatletter
\@namedef{subjclassname@2020}{%
	\textup{2020} Mathematics Subject Classification}
\makeatother

\begin{document}
	
\title[\hfil Equivalence of weak and viscosity solutions\dots] {Equivalence of weak and viscosity solutions for the nonhomogeneous double phase equation}

\author[Y. Fang, V.D. R\u{a}dulescu, C. Zhang]{Yuzhou Fang, Vicen\c{t}iu D. R\u{a}dulescu$^{*, **}$ and Chao Zhang$^{***}$}

\thanks{$^*$Corresponding author: Vicen\c{t}iu D. R\u{a}dulescu ({\tt radulescu@inf.ucv.ro})}
\thanks{$^{**}$ORCID: 0000-0003-4615-5537 (Vicen\c{t}iu D. R\u{a}dulescu)}
\thanks{$^{***}$ORCID: 0000-0003-2702-2050 (Chao Zhang)}

\address{Yuzhou Fang \hfill\break School of Mathematics, Harbin Institute of Technology, Harbin 150001, China}
\email{18b912036@hit.edu.cn}

\address{Vicen\c{t}iu D. R\u{a}dulescu   \hfill\break   Faculty of Applied Mathematics,
	AGH University of Science and Technology,  Krak\'{o}w 30-059, Poland \&   Department of Mathematics, University of Craiova, Craiova 200585, Romania}
\email{radulescu@inf.ucv.ro}

\address{Chao Zhang  \hfill\break School of Mathematics and Institute for Advanced Study in Mathematics, Harbin Institute of Technology, Harbin 150001, China}
\email{czhangmath@hit.edu.cn}


\maketitle

\begin{abstract}
We establish the equivalence between weak and viscosity solutions to the nonhomogeneous double phase equation with lower-order term
$$
-\dive(|Du|^{p-2}Du+a(x)|Du|^{q-2}Du)=f(x,u,Du),\quad 1<p\le q<\infty, a(x)\ge0.
$$
We find some appropriate hypotheses on the coefficient $a(x)$,  the exponents $p, q$ and the nonlinear term $f$ to show that the viscosity solutions with {\em a priori} Lipschitz continuity are weak solutions of such equation by virtue of the $\inf$($\sup$)-convolution techniques. The reverse implication can be concluded through comparison principles. Moreover, we verify that the bounded viscosity solutions are exactly Lipschitz continuous, which is also of independent interest.
\end{abstract}

\section{Introduction}
\label{sec-1}

Let $\Omega$ be a bounded domain in $\mathbb R^n$ ($n\ge 2$). In this work we aim to examine the inner relationship between weak and viscosity solutions to the following nonhomogeneous double phase equation
\begin{equation}
\label{main}
-\dive (|Du|^{p-2}Du+a(x)|Du|^{q-2}Du)=f(x,u,Du) \quad \text{in } \Omega,
\end{equation}
where $1<p\leq q<\infty$, $a(x)\geq0$ and $f(x,\tau,\xi):\Omega\times\mathbb{R}\times\mathbb{R}^n\rightarrow\mathbb{R}$ is a continuous function. The double phase problems, stemming from the models of strongly anisotropic materials, were originally investigated by Zhikov \cite{Zhi86,Zhi93}  and Marcellini \cite{Mar91} in the context of homogenization and Lavrentiev phenomenon.

Over the last years, problems  of the type considered in \eqref{main}  have attracted intensive attention from the variational point of view, whose celebrated prototype is given by the following unbalanced energy functional
$$
W^{1,1}(\Omega)\ni u\mapsto\mathcal{P}(u,\Omega):=\int_\Omega (|Du|^p+a(x)|Du|^q)\,dx.
$$
The significant characteristics of this functional are that its ellipticity and growth rate will change drastically according to the modulating coefficient $a(\cdot)$ equal to 0 or not. The regularity of minimizers is determined via a delicate interaction between the growth conditions and the pointwise behaviour of $a(\cdot)$. For instance, under the hypotheses that
\begin{equation}
\label{2-0}
0\le a(\cdot)\in C^{0,\alpha}(\Omega),  \alpha\in (0,1] \quad\text{and}\quad \frac{q}{p}\le1+\frac{\alpha}{n},
\end{equation}
Colombo, Mingione {\it et. al.} \cite{CM15,BCM18,BCM15} established the gradient H\"{o}lder continuity and Harnack inequality for the minimizers of $\mathcal{P}$. A key feature of this problem is that the minimizers could be even discontinuous when condition \eqref{2-0} is violated, by means of the counterexamples presented in \cite{ELM04,FMM04}.  For the double phase equation
$$
-\dive (|Du|^{p-2}Du+a(x)|Du|^{q-2}Du)=-\dive (|F|^{p-2}F+a(x)|F|^{q-2}F) \quad \text{in } \Omega,
$$
the Calder\'{o}n-Zygmund estimates of weak solutions were derived in \cite{CM16,DeFM} under the assumptions in \eqref{2-0} (see also \cite{BO17, BBO20}). More recently, De Filippis and Mingione  \cite{DeFM21} considered  a very large class of vector-valued nonautonomous variational problems involving integral functionals of the double phase type, where the authors provided a comprehensive treatment of Lipschitz regularity of solutions under sharp conditions. Despite their relatively short history, double phase problems have achieved very fruitful results with several connections to other aspects, such as the existence and multiplicity of solutions \cite{PRR20}, the nonlocal version \cite{DeFP19, FZ210}, the properties of eigenvalues and eigenfunctions \cite{CS16,PPR22}, and the removability and obstacle problems \cite{CDeF20,KL21}. We also refer to \cite{BDS20, DeFM22, FRZZ22, HJ22, L22, Mar21, MR21} and references therein for more results.

The topic on equivalence of different solutions starts from the works of Lions \cite{Lions}  and Ishii \cite{Ish95} on linear equations. For what concerns the quasilinear case, Juutinen, Lindqvist and Manfredi \cite{JLM01} proved that the weak solutions coincide with the viscosity solutions to $p$-Laplace equation and its parabolic version based on the uniqueness machinery of solutions; see \cite{JLP10} for $p(x)$-Laplace type equation. The equivalence of solutions was generalized to the fractional $p$-Laplace equation in \cite{KKP19} by following the analogous ideas. Julin and Juutinen \cite{JJ12} gave a more immediate proof for the equivalence of viscosity and weak solutions to the $p$-Laplace equation without relying on the comparison principle of viscosity solutions. They introduced a technical regularization process through infimal convolution, which was applied to various equations incorporating the normalized $p(x)$-Laplace equation \cite{Sil18}, the nonhomogeneous nonlocal $p$-Laplace equation \cite{BM21} and the normalized $p$-Possion equation \cite{APR17}. More related results can be found in \cite{MO19,Sil21,PV,MO22}.

From the results mentioned above, we can see that the research achievements for the double phase problems mainly focus on the weak solutions from the variational perspective and there are few results concerning the relationship between viscosity and weak solutions for the general nonuniformly elliptic equations. In particular, De Filippis and Palatucci \cite{DeFP19}  showed that the bounded viscosity solutions  of the nonlocal counterpart to \eqref{main} are locally H\"{o}lder continuous.  For the homogeneous case of \eqref{main}, the first and third authors \cite{FZ2020} demonstrated the equivalence between weak and viscosity solutions through introducing $\mathcal{A}_{H(\cdot)}$-harmonic functions that serve as a bridge.  Motivated by the previous works \cite{DeFP19,FZ2020}, our intention in the present paper is to prove the equivalence of weak and viscosity solutions for the nonhomogeneous problem \eqref{main}. Due to the presence of lower-order term, we cannot introduce $\mathcal{A}_{H(\cdot)}$-harmonic functions any more and it is hard to use the full uniqueness machinery of viscosity solutions. To this end, we revisit the inf(sup)-convolution approximation, developed by \cite{JJ12}, to verify directly that weak solutions are equivalent to viscosity solutions under some proper preconditions.

We are now in a position to state the main contributions of this manuscript. The first one is as follows:

\begin{theorem}
\label{thm1}
Let $0<a(x)\in C^1(\Omega)$ and $\frac{q}{p}\le 1+\frac{1}{n}$ be in force. Suppose that $f(x,\tau,\xi)$ is uniformly continuous in $\Omega\times\mathbb{R}\times\mathbb{R}^n$, decreasing in $\tau$, Lipschitz continuous with respect to $\xi$ and fulfill the following growth condition
\begin{equation}
\label{0-2}
|f(x,\tau,\xi)|\le \gamma(|\tau|)(|\xi|^{p-1}+a(x)|\xi|^{q-1})+\Phi(x),
\end{equation}
where $\gamma(\cdot)\ge0$ is continuous and $\Phi\in L^\infty_{\rm loc}(\Omega)$. Let $u$ be a viscosity supersolution with local Lipschitz continuity to \eqref{main} in $\Omega$. Then $u$ is a weak supersolution as well.
\end{theorem}

We would like to mention that the double phase operator, compared to the usual $p$-Laplace operator, lacks translation invariance property and exhibits two diverse growth terms owing to the presence of $a(x)$.  It will lead to an additional error term $E(\varepsilon)$ in the key Lemma \ref{lem3-1} and demand an {\em a priori} assumption that  the viscosity solution $u$ is locally Lipschictz continuous.

The second result is about the local Lipschitz continuity of viscosity solutions, which is also of independent interest.

\begin{theorem}
\label{thm1-1}
Let $u$ be a bounded viscosity solution to \eqref{main} in $\Omega$. Under the assumptions that $0\le a(x)\in C^1(\Omega)$, $p\le q\le p+\frac{1}{2}$ and \eqref{0-2}, for any $\Omega'\subset\subset\Omega$, there is a constant $C$ that depends on $n,p,q,\gamma_\infty,\Omega',\Omega,\|a\|_{C^1(\Omega)}$, $\|u\|_{L^\infty(\Omega)}$  and $\|\Phi\|_{L^\infty(\Omega)}$, such that
$$
|u(x)-u(y)|\le C|x-y|
$$
for all $x,y\in \Omega'$. Here $\gamma_\infty:=\max_{t\in[0,\|u\|_{L^\infty(\Omega)}]} \gamma(t)$.
\end{theorem}

In order to verify Theorem \ref{thm1-1}, we need to utilize the Ishii-Lions methods in \cite{IL90} twice, and adjust carefully the distance $q-p$ in order to get a contradiction. Combining the above two theorems yields that the bounded viscosity solutions are weak solutions with some explicit conditions.

For showing weak solutions are viscosity solutions, we consider a class of functions satisfying comparison principle.

\begin{definition}
\label{0-1}
Suppose that $u$ is a weak supersolution to \eqref{main} in $\Omega'\subset\Omega$. If for any weak subsolution $v$ of \eqref{main} such that $v\le u$ a.e. in $\partial\Omega'$ there holds that $v\le u$ a.e. in $\Omega'$, then we say that $(u,f)$ fulfills comparison principle property (CCP) in $\Omega'$.
\end{definition}

Finally, the weak solutions being viscosity solutions could be obtained under the (CCP) condition by contradiction argument.

\begin{theorem}
\label{thm2}
Let $u$ be a lower semicontinuous weak supersolution to \eqref{main} in $\Omega$. Assume that $f(x,\tau,\xi)$ is uniformly continuous in $\Omega\times\mathbb{R}\times\mathbb{R}^n$. If (CPP) holds true, then $u$ is also a viscosity supersolution to \eqref{main}.
\end{theorem}

\begin{remark}
When the nonlinear term $f$ only depends on $x$, not on $u$ and $\nabla u$,  we know from \cite{DeFM21} that the weak solutions are locally Lipschitz continuous under the minimal hypotheses that $0\le a\in W^{1,d}(\Omega),d>n$, and $f$ belongs to a proper Lorentz space along with $q/p\le 1+1/n-1/d$, if $n>2$ and $q/p<p$, if $n=2$.
\end{remark}

This paper is organized as follows. We in Section \ref{sec-2} introduce some basic properties of function spaces and concepts of solutions as well as some necessary known results. Section \ref{sec-3} is devoted to proving that viscosity solutions are weak solutions to \eqref{main}, and the reverse implication is showed in Section \ref{sec-4}, where we also establish the comparison principle for two equations with different nonlinearities. In Section \ref{sec-5}, we verify the bounded viscosity solutions of \eqref{main} are locally Lipschitz continuous, which is the indispensable element of equivalence.

\section{Preliminaries}
\label{sec-2}

In this section, we summarize some basic properties of Musielak-Orlicz-Sobolev space $W^{1,H(\cdot)}(\Omega)$. These properties can be found in \cite{CGSW21,HH19,Mus83}. In addition, we give the different notions of solutions to Eq. \eqref{main} together with some auxiliary results.

\subsection{Function spaces}

In the rest of this paper, unless otherwise stated, we always assume \eqref{2-0} holds.
For all $x\in\Omega$ and $\xi\in \mathbb{R}^n$, we shall use the notation
\begin{equation}
\label{2-0-1}
H(x,\xi):=|\xi|^p+a(x)|\xi|^q.
\end{equation}
With abuse of notation, we shall also denote $H(x,\xi)$ when $\xi\in\mathbb{R}$. Observe that the generalized Young function $H$ is a Musielak-Orlicz function fulfilling $(\Delta_2)$ and $(\nabla_2)$ conditions.

Let us introduce some important properties, to be used later, of energy density $H$ given by \eqref{2-0-1}.
We will keep on denoting $H(x,t)=t^p+a(x)t^q$ for $t\ge0$, that is, $\xi$ is a non-negative number in \eqref{2-0-1}. By the Fenchel-Young conjugate of $H$, we mean the function $H^*(x,t):=\sup_{s\ge0}\{st-H(x,s)\}$. It is well known that the equivalence
\begin{equation}
\label{2-0-2}
H^*(x,H(x,t)/t)\sim H(x,t) 
\end{equation}
holds up to some constants depending on $p,q$, and moreover the Young's inequality
\begin{equation}
\label{2-0-3}
st\le H^*(x,t)+H(x,s)
\end{equation}
holds for all $x\in\Omega$, $s,t\in[0,+\infty)$.

The Musielak-Orlicz space $L^{H(\cdot)}(\Omega)$ is defined as
$$
L^{H(\cdot)}(\Omega):=\left\{u:\Omega\rightarrow \mathbb{R }  \text{ measurable}:\varrho_H(u)<\infty\right\},
$$
endowed with the norm
$$
\|u\|_{L^{H(\cdot)}(\Omega)}:=\inf\left\{\lambda>0:\varrho_H\left(\frac{u}{\lambda}\right)\leq1\right\},
$$
where $$\varrho_H(u):=\int_\Omega H(x,u)\,dx=\int_\Omega|u|^p+a(x)|u|^q\,dx$$ is called $\varrho_H$-modular.

 The space $L^{H(\cdot)}(\Omega)$ is a separable, uniformly convex Banach space.
From the definitions of $\varrho_H$-modular and norm, we can find
\begin{equation}
\label{2-1}
\min\left\{\|u\|^p_{L^{H(\cdot)}(\Omega)}, \|u\|^q_{L^{H(\cdot)}(\Omega)}\right\}\leq \varrho_H(u)\leq\max\left\{\|u\|^p_{L^{H(\cdot)}(\Omega)}, \|u\|^q_{L^{H(\cdot)}(\Omega)}\right\}.
\end{equation}
It follows from \eqref{2-1} that
$$
\|u_n-u\|_{L^{H(\cdot)}(\Omega)}\rightarrow0 \quad \Longleftrightarrow \quad \varrho_H(u_n-u)\rightarrow0,
$$
which indicates the equivalence of convergence in $\varrho_H$-modular and in norm. For the space $L^{H^*(\cdot)}(\Omega)$ we know that
\begin{align}
\label{2-1-1}
\min\left\{\left(\varrho_{H^*}(u)\right)^\frac{p}{p+1},\left(\varrho_{H^*}(u)\right)^\frac{q}{q+1}\right\}&\leq \|u\|_{L^{H^*(\cdot)}(\Omega)} \nonumber\\
&\le\max\left\{\left(\varrho_{H^*}(u)\right)^\frac{p}{p+1},\left(\varrho_{H^*}(u)\right)^\frac{q}{q+1}\right\}.
\end{align}
If $u\in L^{H(\cdot)}(\Omega)$ and $v\in L^{H^*(\cdot)}(\Omega)$, the following H\"{o}lder inequality
\begin{align}
\label{2-1-2}
\left|\int_\Omega uv\,dx\right|\le 2\|u\|_{L^{H(\cdot)}(\Omega)}\|v\|_{L^{H^*(\cdot)}(\Omega)}
\end{align}
holds.

The Musielak-Orlicz-Sobolev space $W^{1,H(\cdot)}(\Omega)$ is the set of those functions $u\in L^{H(\cdot)}(\Omega)$ satisfying 
$Du\in L^{H(\cdot)}(\Omega)$. We equip the space $W^{1,H(\cdot)}(\Omega)$ with the norm
$$
\|u\|_{W^{1,{H(\cdot)}}(\Omega)}:=\|u\|_{L^{H(\cdot)}(\Omega)}+\|Du\|_{L^{H(\cdot)}(\Omega)}.
$$
The space $W^{1,H(\cdot)}(\Omega)$ is a separable and reflexible Banach space. The local space $W^{1,H(\cdot)}_{\rm loc}(\Omega)$ is composed of those functions belonging to $W^{1,H(\cdot)}(\Omega')$ for any subdomain $\Omega'$ compactly involved in $\Omega$. Finally, we denote by $W^{1,H(\cdot)}_0(\Omega)$ the closure of $C^{\infty}_0(\Omega)$ in $W^{1,H(\cdot)}(\Omega)$. Indeed, the condition \eqref{2-0} ensures that the set $C^\infty_0(\Omega)$ is dense in $W^{1,H(\cdot)}_0(\Omega)$ (see \cite{ELM04,ACGY18}).

\subsection{Notions of solutions}

Set
$$
A(x,\xi):=|\xi|^{p-2}\xi+a(x)|\xi|^{q-2}\xi
$$
for all $x\in \Omega$ and $\xi\in\mathbb{R}^n$. We now give the definition of solutions of diverse type to \eqref{main}.

\begin{definition} [weak solution]
\label{def2-1}
We say that $u\in W^{1,H(\cdot)}_{\rm loc}(\Omega)$ is a weak superslotion to \eqref{main}, if
$$
\int_\Omega\langle A(x,Du), D\phi\rangle\,dx\geq \int_\Omega f(x,u,Du) \phi\,dx
$$
for each nonnegative function $\phi\in W^{1,H(\cdot)}_0(\Omega)$. The inequality is reverse for weak subsolution. When $u\in W^{1,H(\cdot)}_{\rm loc}(\Omega)$ is both weak super- and subsolution, we call $u$ a weak solution to \eqref{main}, that is
$$
\int_\Omega\langle A(x,Du), D\phi\rangle\,dx=\int_\Omega f(x,u,Du) \phi\,dx
$$
for any $\phi\in W^{1,H(\cdot)}_0(\Omega)$.
\end{definition}

Let $\xi,\eta\in\mathbb{R}^n, X\in \mathcal{S}^n$ with $\mathcal{S}^n$ being the set of symmetric $n\times n$ matrices. We introduce some notations
$$
M(x,\xi)=a(x)|\xi|^{q-2}\left(I+(q-2)\frac{\xi}{|\xi|}\otimes\frac{\xi}{|\xi|}\right),
$$
$$
F_1(\xi,X)=-|\xi|^{p-2}\left(\mathrm{tr}X+(p-2)\left\langle X\frac{\xi}{|\xi|},\frac{\xi}{|\xi|}\right\rangle\right),
$$
$$
F_2(x,\xi,X)=-a(x)|\xi|^{q-2}\left(\mathrm{tr}X+(q-2)\left\langle X\frac{\xi}{|\xi|},\frac{\xi}{|\xi|}\right\rangle\right)=-\mathrm{tr}(M(x,\xi)X)
$$
and
$$
F_3(x,\xi)=-|\xi|^{q-2}\xi\cdot Da(x),
$$
where $\xi\otimes\eta$ denotes an $n\times n$ matrix whose $(i,j)$ entry is $\xi_i\eta_j$, and $\langle \xi,\eta\rangle$ or $\xi\cdot\eta$ stands for the inner product of $\xi,\eta$. For a matrix $X$, we set the matrix norm $\|X\|:=\sup_{|\xi|\le1}\{|X\xi|\}$. In order to define the viscosity solutions of \eqref{main}, we let $a\in C^1(\Omega)$ and easily check that
\begin{align}
\label{double}
-\dive(|Du|^{p-2}Du+a(x)|Du|^{q-2}Du)&=F_1(Du,D^2u)+F_2(x,Du,D^2u)+F_3(x,Du) \nonumber\\
&=:F(x,Du,D^2u).
\end{align}

We now recall the notion of semi-jets. The subjet of $u:\Omega\rightarrow\mathbb{R}$ at $x$ is given by letting $(\eta,X)\in J^{2,-}u(x)$ if
$$
u(y)\ge u(x)+\eta\cdot(y-x)+\frac{1}{2}\langle X(y-x),(y-x)\rangle+\mathfrak{o}(|y-x|^2)
$$
as $y\rightarrow x$. The closure of a subjet is defined by $(\eta,X)\in \overline{J}^{2,-}u(x)$ if there exists a sequence $(\eta_j,X_j)\in J^{2,-}u(x_j)$ such that $(x_j,\eta_j,X_j)\rightarrow(x,\eta,X)$. The superjet $J^{2,+}$ and its closure $\overline{J}^{2,+}$ are defined by a similar way but the above inequality needs to be converse.

\begin{definition}[viscosity solution]
\label{def2-3}
A lower semicontinuous function $u:\Omega\rightarrow(-\infty,\infty)$ is a viscosity supersolution to \eqref{main} in $\Omega$, if $(\eta,X)\in J^{2,-}u(x)$ with $x\in\Omega$ and $\eta\neq0$ implies that
$$
F(x,\eta,X)\ge f(x,u(x),\eta).
$$
An upper semicontinuous function $u:\Omega\rightarrow(-\infty,\infty)$ is a viscosity subsolution to \eqref{main} in $\Omega$, if for each $(\eta,X)\in J^{2,+}u(x)$ with $x\in\Omega$ and $\eta\neq0$ there holds that
$$
F(x,\eta,X)\le f(x,u(x),\eta).
$$
A function $u$ is called viscosity solution to \eqref{main} if and only if it is viscosity super- and subsolution.
\end{definition}

\begin{remark}
The preceding concept of viscosity solutions is equivalently given by the jet-closures or test functions. For instance, the following conditions are equivalent:
\begin{itemize}
\item[(1)] A function $u$ is a viscosity supersolution to \eqref{main} in $\Omega$;

\smallskip

\item[(2)] If $(\eta,X)\in \overline{J}^{2,-}u(x)$ with $x\in\Omega$ and $\eta\neq0$, then $F(x,\eta,X)\ge f(x,u(x),\eta)$;

\smallskip

\item[(3)] If $\varphi\in C^2(\Omega)$ touches $u$ from below at $x$, that is, $\varphi(x)=u(x), \varphi(y)\le u(y)$ and moreover $D\varphi(x)\neq0$, then we have $F(x,D\varphi(x),D^2\varphi(x))\ge f(x,u(x),D\varphi(x))$.
\end{itemize}
In the case $2\le p\leq q$, we can remove the requirement that $\eta\neq0$ or $D\varphi(x)\neq0$.
\end{remark}

\subsection{Inf-convolution}

We now give the definition together with some properties of infimal convolution. Define the inf-convolution as
\begin{equation*}
u_\varepsilon(x)=\inf_{y\in\Omega} \left\{u(y)+\frac{|x-y|^s}{s\varepsilon^{s-1}}\right\},
\end{equation*}
where $\varepsilon>0$ and $s\geq\max\left\{2,\frac{p}{p-1}\right\}$ is a constant to be fixed by the growth powers in Eq. \eqref{main}. Indeed, when $2\leq p\leq q$, $s=2$; when $1<p\leq q<2$, $s>\max\left\{\frac{p}{p-1},\frac{q}{q-1}\right\}=\frac{p}{p-1}$; when $1<p<2\leq q$, $s>\max\left\{\frac{p}{p-1},2\right\}=\frac{p}{p-1}$.

The following well-known properties of the inf-convolution $u_\varepsilon$ can be found in much literature, such as \cite{JJ12,Sil18}.

\begin{proposition}
\label{pro2-1}

Suppose that $u:\Omega\rightarrow\mathbb{R}$ is a bounded and lower semicontinuous function. Then the inf-convolution $u_\varepsilon$ satisfies the following properties:
\begin{itemize}
\item[(1)] $u_\varepsilon\leq u$ in $\Omega$ and $u_\varepsilon\rightarrow u$ locally uniformly as $\varepsilon\rightarrow0$.

\smallskip

\item[(2)] There is $r(\varepsilon)>0$ such that
    $$
    u_\varepsilon(x)=\inf_{y\in{B_{r(\varepsilon)}(x)\cap\Omega}} \left\{u(y)+\frac{|x-y|^s}{s\varepsilon^{s-1}}\right\}
    $$
    with $r(\varepsilon)\rightarrow0$ as $\varepsilon\rightarrow0$. In particular, if $x\in\Omega_{r(\varepsilon)}:=\{x\in\Omega:\mathrm{dist}(x,\partial\Omega)>r(\varepsilon)\}$, then there exists a point $x_\varepsilon\in B_{r(\varepsilon)}(x)$ fulfilling
    $$
    u_\varepsilon(x)=u(x_\varepsilon)+\frac{|x-x_\varepsilon|^s}{s\varepsilon^{s-1}}.
    $$

\smallskip

\item[(3)] The function $u_\varepsilon$ is semi-concave in $\Omega_{r(\varepsilon)}$, that is, we can find a constant $C$, depending only on $u, s$ and $\varepsilon$, such that the function $x\mapsto u_\varepsilon(x)-C|x|^2$ is concave.

\smallskip

\item[(4)] If $(\eta,X)\in J^{2,-}u_\varepsilon(x)$ with $x\in\Omega_{r(\varepsilon)}$, then we have
$$
\eta=\frac{|x-x_\varepsilon|^{s-2}(x-x_\varepsilon)}{\varepsilon^{s-1}} \quad\text{and}\quad X\le \frac{s-1}{\varepsilon}|\eta|^\frac{s-2}{s-1}I.
$$
\end{itemize}
\end{proposition}

\section{Viscosity solutions are weak solutions}
\label{sec-3}

In this part, we are going to make use of inf-convolution approximation technique to show that the locally Lipschitz continuous viscosity solutions are weak solutions to \eqref{main}. Furthermore, the Lipschitz continuity of viscosity solutions can be proved precisely, whose proof is postponed to Section \ref{sec-5}. We therefore draw a conclusion that viscosity solutions are weak solutions under some suitable conditions. We will only discuss viscosity and weak supersolutions below. The case of subsolution is similar.

We begin with stating that if $u$ is a (local Lipschitz) viscosity supersolution to \eqref{main} in $\Omega$, then its inf-convolution $u_\varepsilon$ also is a viscosity supersolutions of such equation (whose form may be slightly changed) in a shrinking domain. From the following lemma, we can find that owing to the presence of modulating coefficient $a(x)$, there will exist an error term $E(\varepsilon)$ on the right-hand side of equation. In what follows, $X\le Y$ ($X,Y\in\mathcal{S}^n$) means that $\langle(X-Y)\xi,\xi\rangle\leq0$ for any $\xi\in \mathbb{R}^n$. We denote by $C$ a generic constant, which may vary from line to line. If necessary, relevant dependencies on parameters will be emphasised using parentheses.

\begin{lemma}
\label{lem3-1}
Assume that $0<a(x)\in C^1(\Omega)$ and $f(x,\tau,\xi)$ is continuous in $\Omega\times\mathbb{R}\times\mathbb{R}^n$ and decreasing with respect to $\tau$. Let $u$ be a viscosity supersolution with local Lipschitz to \eqref{main} in $\Omega$. Then if $(\eta,X)\in J^{2,-}u_\varepsilon(x)$ with $\eta\neq0$ and $x\in\Omega_{r(\varepsilon)}$, there holds that
\begin{equation}
\label{3-1}
F(x,\eta,X)\ge f_\varepsilon(x,u_\varepsilon(x),\eta)+E(\varepsilon),
\end{equation}
where
$$
f_\varepsilon(x,\tau,\xi):=\inf_{y\in B_{r(\varepsilon)}(x)}f(y,\tau,\xi)
$$
and $E(\varepsilon)\rightarrow0$ as $\varepsilon\rightarrow0$.
\end{lemma}

\begin{remark}
It is worth mentioning that the requirement $a(x)>0$ is just needed in this lemma for the technical reason.
\end{remark}

\begin{proof}
Let $x^o\in \Omega_{r(\varepsilon)}$ and $(\eta,X)\in J^{2,-}u_\varepsilon(x^o)$ with $\eta\neq0$. Then via the properties of inf-convolution $u_\varepsilon$ in Proposition \ref{pro2-1}, we have
$$
u_\varepsilon(x^o)=u(x^o_\varepsilon)+\frac{|x^o-x^o_\varepsilon|^s}{s\varepsilon^{s-1}}
$$
with $x^o_\varepsilon\in B_{r(\varepsilon)}(x^o)$, and moreover $\eta=\frac{|x^o-x^o_\varepsilon|^{s-2}}{\varepsilon^{s-1}}(x^o-x^o_\varepsilon)$. There is a function $\varphi\in C^2(\Omega)$ such that it touches $u_\varepsilon$ from below at $x^o$  and $D\varphi(x^o)=\eta, D^2\varphi(x^o)=X$. Hence by the definition of inf-convolution $u_\varepsilon$ we can see that
$$
0\leq u_\varepsilon(x)-\varphi(x)\le u(y)+\frac{|x-y|^s}{s\varepsilon^{s-1}}-\varphi(x)
$$
for any $x,y\in \Omega_{r(\varepsilon)}$. Notice also that
$$
u(x^o_\varepsilon)+\frac{|x^o-x^o_\varepsilon|^s}{s\varepsilon^{s-1}}-\varphi(x^o)=u_\varepsilon(x^o)-\varphi(x^o)=0.
$$
Then the function $-u(y)+\varphi(x)-\frac{|y-x|^s}{s\varepsilon^{s-1}}$ attains the maximum
at $(x^o_\varepsilon,x^o)$. Therefore, by applying theorem of sums in \cite{CIL92}, we can find $Y,Z\in\mathcal{S}^n$ satisfying
$$
(-D_y\psi(x_\varepsilon^o,x^o),-Y)\in \overline{J}^{2,-}u(x_\varepsilon^o), \quad (D_x\psi(x_\varepsilon^o,x^o),-Z)\in \overline{J}^{2,+}\varphi(x^o)
$$
and
\begin{equation*}
\left(\begin{array}{cc}
Y & \\[2mm]
 &-Z
\end{array}
\right)\\
\leq D^2\psi(x_\varepsilon^o,x^o)+\varepsilon^{1-s}(D^2\psi(x_\varepsilon^o,x^o))^2,
\end{equation*}
where $\psi(y,x):=\frac{|y-x|^s}{s\varepsilon^{s-1}}$ and
\begin{equation*}
D^2\psi(x_\varepsilon^o,x^o)=\left(\begin{array}{cc}
D_{yy}\psi(x_\varepsilon^o,x^o) &D_{yx}\psi(x_\varepsilon^o,x^o) \\[2mm]
D_{xy}\psi(x_\varepsilon^o,x^o) &D_{xx}\psi(x_\varepsilon^o,x^o)
\end{array}\right).
\end{equation*}
Via direct computation,
$$
-D_y\psi(x_\varepsilon^o,x^o)=\eta=D_x\psi(x_\varepsilon^o,x^o)
$$
and
$$
B:=D_{yy}\psi(x_\varepsilon^o,x^o)=\varepsilon^{1-s}|x_\varepsilon^o-x^o|^{s-4}[|x_\varepsilon^o-x^o|^2I+(s-2)(x_\varepsilon^o-x^o)\otimes (x_\varepsilon^o-x^o)].
$$
Furthermore,
\begin{equation}
\label{3-2}
\left(\begin{array}{cc}
Y & \\[2mm]
 &-Z
\end{array}
\right)\\[2mm]
\leq \left(\begin{array}{cc}
B &-B\\[2mm]
-B &B
\end{array}\right)+2\varepsilon^{1-s}\left(\begin{array}{cc}
B^2 &-B^2\\[2mm]
-B^2 &B^2
\end{array}\right).
\end{equation}
It is straightforward to derive that $Y\leq Z$ and moreover $Z\le-X$. Since $u$ is a viscosity supersolution to \eqref{main} and $f$ is continuous in all variables, then
\begin{align*}
f(x_\varepsilon^o,u(x_\varepsilon^o),\eta)&\le F(x_\varepsilon^o,\eta,-Y)\\
&=F_1(\eta,Z)-F_1(\eta,Y)-F_1(\eta,Z)+F_2(x^o,\eta,Z)-F_2(x_\varepsilon^o,\eta,Y)\\
&\quad-F_2(x^o,\eta,Z)+F_3(x^o,\eta)+F_3(x^o_\varepsilon,\eta)-F_3(x^o,\eta)\\
&\le F(x^o,\eta,-Z)+F_2(x^o,\eta,Z)-F(x_\varepsilon^o,\eta,Y)+F_3(x^o_\varepsilon,\eta)-F_3(x^o,\eta)
\end{align*}
with the notations given in \eqref{double}, where in the last line we have used the decreasing property of $F_1(\xi,X)$ in the $X$-variable, that is, $Y\le Z$ implies $F_1(\eta,Z)\le F_1(\eta,Y)$. Next, in view of $a\in C^1(\Omega)$ we estimate
\begin{align}
\label{3-3}
F_3(x^o_\varepsilon,\eta)-F_3(x^o,\eta)&=|\eta|^{q-2}\eta\cdot Da(x^o)-|\eta|^{q-2}\eta\cdot Da(x^o_\varepsilon) \nonumber\\
&\le |\eta|^{q-1}|Da(x^o)-Da(x^o_\varepsilon)| \nonumber\\
&\le |\eta|^{q-1}\omega(r(\varepsilon)),
\end{align}
where $\omega(\cdot)$ represents the modulus of continuity of $Da$. We finally evaluate the term $F_2(x^o,\eta,Z)-F_2(x_\varepsilon^o,\eta,Y)$ in a similar way to address $F_2(y_j,\eta_j,Y_j)-F_2(x_j,\eta_j,X_j)$ in \cite[Proposition 5.1]{FZ2020}. Observe that it follows from \eqref{3-2} that
\begin{equation}
\label{3-4}
\langle \xi,Y\xi\rangle-\langle \zeta,Z\zeta\rangle\le \varepsilon^{1-s}\left[(s-1)|x^o_\varepsilon-x^o|^{s-2}+2(s-1)^2|x^o_\varepsilon-x^o|^{2(s-2)}\right]|x^o_\varepsilon-x^o|^2
\end{equation}
with $\xi,\zeta\in\mathbb{R}^n$. We can easily verify the matrix $M(x,\xi)\ge0$ (positive semi-definite) as $a(x)\ge0$ so that it has matrix square root denoted by $M^\frac{1}{2}(x,\xi)$. Additionally, by $M^\frac{1}{2}_l(x,\xi)$ we mean the $l$-th column of $M^\frac{1}{2}(x,\xi)$. Then employing \eqref{3-4} and decomposition of matrix yields that
\begin{align*}
& \quad F_2(x^o,\eta,Z)-F_2(x_\varepsilon^o,\eta,Y)\\
&=\mathrm{tr}\left(M^\frac{1}{2}(x^o,\eta)M^\frac{1}{2}(x^o,\eta)Z\right)-\mathrm{tr}\left(M^\frac{1}{2}(x^o_\varepsilon,\eta)M^\frac{1}{2}(x^o_\varepsilon,\eta)Y\right)\\
&=\sum^n_{l=1}\left\langle M^\frac{1}{2}_l(x^o,\eta),Z M^\frac{1}{2}_l(x^o,\eta)\right\rangle-
\sum^n_{l=1}\left\langle M^\frac{1}{2}_l(x^o_\varepsilon,\eta),YM^\frac{1}{2}_l(x^o_\varepsilon,\eta)\right\rangle\\
&\leq C\varepsilon^{1-s}|x^o_\varepsilon-x^o|^{s-2}\left\|M^\frac{1}{2}(x^o,\eta)-M^\frac{1}{2}(x^o_\varepsilon,\eta)\right\|^2_2\\
&\leq \frac{C\varepsilon^{1-s}|x^o_\varepsilon-x^o|^{s-2}}{\left(\lambda_{\mathrm{min}}\left(M^\frac{1}{2}(x^o,\eta)\right)
+\lambda_{\mathrm{min}}\left(M^\frac{1}{2}(x^o_\varepsilon,\eta)\right)\right)^2}\|M(x^o,\eta)-M(x^o_\varepsilon,\eta)\|^2_2\\
&\le \frac{C\varepsilon^{1-s}|x^o_\varepsilon-x^o|^{s-2}|\eta|^{2(q-2)}|a(x^o)-a(x_\varepsilon^o)|^2}{|\eta|^{q-2}\left(\sqrt{a(x^o)}+\sqrt{a(x^o_\varepsilon)}\right)^2},
\end{align*}
where $\lambda_{\mathrm{min}}(M)$ stands for the smallest eigenvalue of the matrix $M$. For more details of the above display can be found in  \cite[Proposition 5.1]{FZ2020}.

On the other hand, from the local Lipschitz continuity of $u$,
$$
\frac{|x^o-x^o_\varepsilon|^s}{s\varepsilon^{s-1}}=u_\varepsilon(x^o)-u(x^o_\varepsilon)<u(x^o)-u(x^o_\varepsilon)\le C|x^o_\varepsilon-x^o|,
$$
i.e.,
$$
|\eta|=\frac{|x^o-x^o_\varepsilon|^{s-1}}{\varepsilon^{s-1}}\le C.
$$
By means of $a(x)\in C^1(\Omega)$, we proceed to treat
\begin{align*}
F_2(x^o,\eta,Z)-F_2(x_\varepsilon^o,\eta,Y)&\le \frac{C\varepsilon^{1-s}|x^o_\varepsilon-x^o|^{s-2}|\eta|^{q-2}|x^o-x_\varepsilon^o|^2}{\left(\sqrt{a(x^o)}+\sqrt{a(x^o_\varepsilon)}\right)^2}\\
&= \frac{C|\eta|^{q-1}|x^o-x_\varepsilon^o|}{\left(\sqrt{a(x^o)}+\sqrt{a(x^o_\varepsilon)}\right)^2}\\
&\le \frac{Cr(\varepsilon)}{\left(\sqrt{a(x^o)}+\sqrt{a(x^o_\varepsilon)}\right)^2}.
\end{align*}
Here we remark that if $a(x^o)=0$, then the quantity $\frac{|\eta|^{q-1}|x^o-x_\varepsilon^o|}{a(x^o_\varepsilon)}$ does not necessarily go to 0 as $\varepsilon\rightarrow0$, so the condition $a(x)>0$ is required. Besides, by the boundedness of $\eta$ the inequality \eqref{3-3} becomes
$$
F_3(x^o_\varepsilon,\eta)-F_3(x^o,\eta)\le C\omega(r(\varepsilon)).
$$
Since $f(x,\tau,\xi)$ is decreasing in $\tau$ and $u(x^o_\varepsilon)\le u_\varepsilon(x^o)$,
$$
f(x^o_\varepsilon,u(x^o_\varepsilon),\eta)\ge f(x^o_\varepsilon,u_\varepsilon(x^o),\eta)\ge \inf_{y\in B_{r(\varepsilon)}(x^o)}f(y,u_\varepsilon(x^o),\eta).
$$
Now define
$$
E(\varepsilon):=C\omega(r(\varepsilon))+\frac{Cr(\varepsilon)}{\left(\sqrt{a(x^o)}+\sqrt{a(x^o_\varepsilon)}\right)^2}.
$$
Consequently, we get
$$
f_\varepsilon(x^o,u_\varepsilon(x^o),\eta)+E(\varepsilon)\le F(x^o,\eta,X)
$$
with $E(\varepsilon)\rightarrow0$ as $\varepsilon\rightarrow0$. Here we have employed $F(x^o,\eta,X)\ge F(x^o,\eta,-Z)$ by $Z\le-X$ and the error term $E(\varepsilon)$ depends on $\varepsilon,n,p,q,a$ and the Lipschitz constant of $u$. The proof is now completed.
\end{proof}

Based on the preceding lemma, we further demonstrate that when $u$ is a viscosity supersolution to \eqref{main}, then the inf-convolution $u_\varepsilon$ is a weak supersolution of this equation up to a certain error term. Let $\Delta_\infty u=\langle Du,D^2uDu\rangle$ below.

\begin{lemma}
\label{lem3-2}
Suppose that the assumptions on $a, f$ in Lemma \ref{lem3-1} are in force, and that $f(x,\tau,0)\le0$ for all $(x,\tau)\in \Omega\times\mathbb{R}$. Let $u$ be a locally Lipschitz continuous viscosity supersolution to \eqref{main}. Then, for each nonnegative function $\varphi\in W^{1,H(\cdot)}_0(\Omega_{r(\varepsilon)})$, it holds that
$$
\int_{\Omega_{r(\varepsilon)}}\varphi f_\varepsilon(x,u_\varepsilon,Du_\varepsilon)\,dx+E(\varepsilon)\int_{\{Du_\varepsilon\neq0\}}\varphi\,dx\leq
\int_{\Omega_{r(\varepsilon)}}\langle A(x,Du_\varepsilon),D\varphi\rangle\,dx,
$$
where $E(\varepsilon)\rightarrow0$, which is from Lemma \ref{lem3-1}, as $\varepsilon\rightarrow0$.
\end{lemma}

\begin{proof}
It suffices to consider the nonnegative function $\varphi\in C^\infty_0(\Omega_{r(\varepsilon)})$, because the function space $C^\infty_0(\Omega_{r(\varepsilon)})$ is dense in $W^{1,H(\cdot)}_0(\Omega_{r(\varepsilon)})$.

Owing to $u_\varepsilon$ being semi-concave, we can see by Proposition \ref{pro2-1} that
$$
h(x):=u_\varepsilon(x)-C(s,\varepsilon,u)|x|^2
$$
is concave in $\Omega_{r(\varepsilon)}$. Let $\{h_j\}_j$ be a sequence of smooth concave functions, obtained from standard mollification, such that
\begin{equation*}
(h_j,Dh_j,D^2h_j)\rightarrow(h,Dh,D^2h) \quad \text{a.e. in } \Omega_{r(\varepsilon)}.
\end{equation*}
Additionally, we define
$$
u_{\varepsilon,j}=h_j+C(s,\varepsilon,u)|x|^2
$$
and let $\delta\in(0,1)$. In this proof, the condition $0\le a(x)\in C^{0,1}(\Omega)$ is sufficient except applying Lemma \ref{lem3-1}. Now denote the standard mollification of $a$ as $a_j$.

\textbf{Case 1.} $1<p\le q<2$. In this singular case, we first regularize the equation by adding a small $\delta>0$ as follows, and eventually pass to the limit as $\delta\rightarrow0$. Recalling that $u_{\varepsilon,j}$ and $a_j$ are smooth, we can calculate by integration by parts
\begin{align}
\label{3-5}
&\quad\int_{\Omega_{r(\varepsilon)}}-\varphi\dive\left[(|Du_{\varepsilon,j}|^2+\delta)^\frac{p-2}{2}Du_{\varepsilon,j}+a_j(x)(|Du_{\varepsilon,j}|^2+\delta)^\frac{q-2}{2}
Du_{\varepsilon,j}\right]\,dx \nonumber\\
&=\int_{\Omega_{r(\varepsilon)}}\left\langle(|Du_{\varepsilon,j}|^2+\delta)^\frac{p-2}{2}Du_{\varepsilon,j}+a_j(x)(|Du_{\varepsilon,j}|^2+\delta)^\frac{q-2}{2}
Du_{\varepsilon,j},D\varphi\right\rangle\,dx.
\end{align}
We are ready to consider the limit as $j\rightarrow\infty$, and claim that
\begin{align}
\label{3-6}
&\quad-\int_{\Omega_{r(\varepsilon)}}\varphi\dive\left[(|Du_{\varepsilon}|^2+\delta)^\frac{p-2}{2}Du_{\varepsilon}+a(x)(|Du_{\varepsilon}|^2+\delta)^\frac{q-2}{2}
Du_{\varepsilon}\right]\,dx \nonumber\\
&\le\int_{\Omega_{r(\varepsilon)}}\left\langle(|Du_{\varepsilon}|^2+\delta)^\frac{p-2}{2}Du_{\varepsilon}+a(x)(|Du_{\varepsilon}|^2+\delta)^\frac{q-2}{2}
Du_{\varepsilon},D\varphi\right\rangle\,dx.
\end{align}
It follows, from the Lipschitz continuity of $u_\varepsilon$ and $a$, that for some constant $M>0$
$$
\|Du_{\varepsilon,j}\|_{L^\infty(\rm supp\,\varphi)}, \ \|a_j\|_{L^\infty(\rm supp\,\varphi)},\ \|Da_j\|_{L^\infty(\rm supp\,\varphi)}\le M,  \quad j=1,2,3,\cdots.
$$
Thus we could apply the Lebesgue dominated convergence theorem to the integral at the right-hand side of \eqref{3-5}. On the other hand, via direct computation,
\begin{align}
\label{3-7}
&\quad-\int_{\Omega_{r(\varepsilon)}}\varphi\dive\left[(|Du_{\varepsilon,j}|^2+\delta)^\frac{p-2}{2}Du_{\varepsilon,j}+a_j(x)(|Du_{\varepsilon,j}|^2+\delta)^\frac{q-2}{2}
Du_{\varepsilon,j}\right]\,dx \nonumber\\
&=-\int_{\Omega_{r(\varepsilon)}}\varphi(|Du_{\varepsilon,j}|^2+\delta)^\frac{p-2}{2}\left(\Delta u_{\varepsilon,j}+\frac{p-2}{|Du_{\varepsilon,j}|^2+\delta}\Delta_\infty u_{\varepsilon,j}\right)\,dx \nonumber\\
&\quad-\int_{\Omega_{r(\varepsilon)}}\varphi a_j(|Du_{\varepsilon,j}|^2+\delta)^\frac{q-2}{2}\left(\Delta u_{\varepsilon,j}+\frac{q-2}{|Du_{\varepsilon,j}|^2+\delta}\Delta_\infty u_{\varepsilon,j}\right)\,dx \nonumber\\
&\quad-\int_{\Omega_{r(\varepsilon)}}\varphi (|Du_{\varepsilon,j}|^2+\delta)^\frac{q-2}{2}Du_{\varepsilon,j}\cdot Da_j\,dx.
\end{align}
Obviously, by the dominated convergence theorem, when $j\rightarrow\infty$,
$$
\int_{\Omega_{r(\varepsilon)}}\varphi (|Du_{\varepsilon,j}|^2+\delta)^\frac{q-2}{2}Du_{\varepsilon,j}\cdot Da_j\,dx\rightarrow
\int_{\Omega_{r(\varepsilon)}}\varphi (|Du_{\varepsilon}|^2+\delta)^\frac{q-2}{2}Du_{\varepsilon}\cdot Da\,dx.
$$
Note that $h_j$ is concave. Then we have $D^2u_{\varepsilon,j}\leq C(s,\varepsilon,u)I$. For $Du_{\varepsilon,j}\neq0$, we arrive at
\begin{align*}
(|Du_{\varepsilon,j}|^2+\delta)^\frac{p-2}{2}\left(\Delta u_{\varepsilon,j}+\frac{p-2}{|Du_{\varepsilon,j}|^2+\delta}\Delta_\infty u_{\varepsilon,j}\right)\le C(s,\varepsilon,u)\delta^\frac{p-2}{2}(2n+p-2)
\end{align*}
and
\begin{align*}
a_j(x)(|Du_{\varepsilon,j}|^2+\delta)^\frac{q-2}{2}\left(\Delta u_{\varepsilon,j}+\frac{q-2}{|Du_{\varepsilon,j}|^2+\delta}\Delta_\infty u_{\varepsilon,j}\right)\le MC(s,\varepsilon,u)\delta^\frac{q-2}{2}(2n+q-2).
\end{align*}
For $Du_{\varepsilon,j}=0$, the case becomes easier. Therefore, we can exploit Fatou's lemma to the display \eqref{3-7}, and further justify \eqref{3-6}.

Next, we shall let $\delta\rightarrow0$ in the inequality \eqref{3-6}. It follows from the dominated convergence theorem that, when $\delta$ goes to 0,
\begin{equation*}
\begin{split}
&\quad\int_{\Omega_{r(\varepsilon)}}\left\langle(|Du_{\varepsilon}|^2+\delta)^\frac{p-2}{2}Du_{\varepsilon}+a(x)(|Du_{\varepsilon}|^2+\delta)^\frac{q-2}{2}
Du_{\varepsilon},D\varphi\right\rangle\,dx\\
&\rightarrow\int_{\Omega_{r(\varepsilon)}}\left\langle|Du_{\varepsilon}|^{p-2}Du_{\varepsilon}+a(x)|Du_{\varepsilon}|^{q-2}
Du_{\varepsilon},D\varphi\right\rangle\,dx.
\end{split}
\end{equation*}
Besides, we show that the integrand at the left-hand side of \eqref{3-6} is bounded from below, which can justify the use of Fatou's lemma. If $Du_\varepsilon=0$, this follows immediately from the inequality (see Proposition \ref{pro2-1})
$$
D^2u_\varepsilon\le \frac{s-1}{\varepsilon}|Du_\varepsilon|^\frac{s-2}{s-1}I.
$$
In other words, since $s>2$, then $D^2u_\varepsilon$ is negative semi-definite when $Du_\varepsilon=0$. If $Du_\varepsilon\neq0$, we will find that
\begin{align*}
&\quad -(|Du_{\varepsilon}|^2+\delta)^\frac{p-2}{2}\left(\Delta u_{\varepsilon}+\frac{p-2}{|Du_{\varepsilon}|^2+\delta}\Delta_\infty u_{\varepsilon}\right)\\
&\ge -\frac{(|Du_{\varepsilon}|^2+\delta)^\frac{p-2}{2}}{|Du_{\varepsilon}|^2+\delta}\frac{s-1}{\varepsilon}\left(|Du_\varepsilon|^{\frac{s-2}{s-1}+2}(n+p-2)
+\delta n|Du_\varepsilon|^\frac{s-2}{s-1}\right)\\
&\ge -|Du_\varepsilon|^{\frac{s-2}{s-1}+p-2}\frac{s-1}{\varepsilon}(2n+p-2)\\
&\ge -\|Du_\varepsilon\|^{\frac{s-2}{s-1}+p-2}_{L^\infty(\Omega_{r(\varepsilon)})}\frac{s-1}{\varepsilon}(2n+p-2),
\end{align*}
where in the last inequality we need to recall the local Lipschitz continuity of $u_\varepsilon$ and $s>\frac{p}{p-1}$. Likewise,
\begin{align*}
&\quad -a(x)(|Du_{\varepsilon}|^2+\delta)^\frac{q-2}{2}\left(\Delta u_{\varepsilon}+\frac{q-2}{|Du_{\varepsilon}|^2+\delta}\Delta_\infty u_{\varepsilon}\right)\\
&\ge -\|a\|_{L^\infty(\Omega_{r(\varepsilon)})}\|Du_\varepsilon\|^{\frac{s-2}{s-1}+q-2}_{L^\infty(\Omega_{r(\varepsilon)})}\frac{s-1}{\varepsilon}(2n+q-2).
\end{align*}
Here we note $s>\frac{p}{p-1}\ge \frac{q}{q-1}$. Moreover, we have
$$
\left|(|Du_{\varepsilon}|^2+\delta)^\frac{q-2}{2}Du_{\varepsilon}\cdot Da\right|\le \|Da\|_{L^\infty(\Omega_{r(\varepsilon)})}\|Du_\varepsilon\|^{q-1}_{L^\infty(\Omega_{r(\varepsilon)})}.
$$

On the other hand, we know that $(Du_\varepsilon(x),D^2u_\varepsilon(x))\in J^{2,-}u_\varepsilon(x)$ for almost every $x\in\Omega_{r(\varepsilon)}$. Then by means of Lemma \ref{lem3-1} we deduce
\begin{equation}
\label{3-8}
F(x,Du_\varepsilon(x),D^2u_\varepsilon(x))\ge f_\varepsilon(x,u_\varepsilon(x),Du_\varepsilon(x))+E(\varepsilon)
\end{equation}
in $\{x\in\Omega_{r(\varepsilon)}:Du_\varepsilon\neq0\}$. As a consequence, exploiting Fatou's lemma along with observing that $D^2u_\varepsilon$ is negative semi-definite when $Du_\varepsilon=0$, we derive from \eqref{3-6} that
\begin{align}
\label{3-9}
&\quad \int_{\Omega_{r(\varepsilon)}}\left\langle|Du_{\varepsilon}|^{p-2}Du_{\varepsilon}+a(x)|Du_{\varepsilon}|^{q-2}
Du_{\varepsilon},D\varphi\right\rangle\,dx \nonumber\\
&\ge \liminf_{\delta\rightarrow0}\left[\int_{\{Du_\varepsilon\neq0\}}+\int_{\{Du_\varepsilon=0\}}-\varphi
(|Du_{\varepsilon}|^2+\delta)^\frac{p-2}{2}\left(\Delta u_{\varepsilon}+\frac{p-2}{|Du_{\varepsilon}|^2+\delta}\Delta_\infty u_{\varepsilon}\right)\,dx\right] \nonumber\\
&\quad+\liminf_{\delta\rightarrow0}\left[\int_{\{Du_\varepsilon\neq0\}}+\int_{\{Du_\varepsilon=0\}}-\varphi a(x)
(|Du_{\varepsilon}|^2+\delta)^\frac{q-2}{2}\left(\Delta u_{\varepsilon}+\frac{q-2}{|Du_{\varepsilon}|^2+\delta}\Delta_\infty u_{\varepsilon}\right)\,dx\right] \nonumber\\
&\quad+\liminf_{\delta\rightarrow0}\left[\int_{\{Du_\varepsilon\neq0\}}+\int_{\{Du_\varepsilon=0\}}-\varphi
(|Du_{\varepsilon}|^2+\delta)^\frac{q-2}{2}Du_{\varepsilon}\cdot Da\,dx\right] \nonumber\\
&\ge \int_{\{Du_\varepsilon\neq0\}}-\varphi
|Du_{\varepsilon}|^{p-2}\left(\Delta u_{\varepsilon}+\frac{p-2}{|Du_{\varepsilon}|^2}\Delta_\infty u_{\varepsilon}\right)\,dx  \nonumber\\
&\quad+\int_{\{Du_\varepsilon\neq0\}}-\varphi a(x)
|Du_{\varepsilon}|^{q-2}\left(\Delta u_{\varepsilon}+\frac{q-2}{|Du_{\varepsilon}|^2}\Delta_\infty u_{\varepsilon}\right)\,dx  \nonumber\\
&\quad+\int_{\{Du_\varepsilon\neq0\}}-\varphi
|Du_{\varepsilon}|^{q-2}Du_\varepsilon\cdot Da\,dx  \nonumber\\
&=\int_{\{Du_\varepsilon\neq0\}}\varphi F(x,Du_\varepsilon,D^2u_\varepsilon)\,dx  \nonumber\\
&\ge \int_{\{Du_\varepsilon\neq0\}}\varphi f_\varepsilon(x,u_\varepsilon,Du_\varepsilon)\,dx+E(\varepsilon)\int_{\{Du_\varepsilon\neq0\}}\varphi\,dx \nonumber\\
&\ge\int_{\Omega_{r(\varepsilon)}}\varphi f_\varepsilon(x,u_\varepsilon,Du_\varepsilon)\,dx+E(\varepsilon)\int_{\{Du_\varepsilon\neq0\}}\varphi\,dx,
\end{align}
where in the penultimate line we employed the inequality \eqref{3-8}, and in the last line we used the assumption that $f(x,\tau,0)\le0$.

\textbf{Case 2.} $1<p<2\leq q$. We need to substitute the identity \eqref{3-5} with
\begin{align*}
&\quad\int_{\Omega_{r(\varepsilon)}}-\varphi\dive\left[(|Du_{\varepsilon,j}|^2+\delta)^\frac{p-2}{2}Du_{\varepsilon,j}+a_j(x)|Du_{\varepsilon,j}|^{q-2}
Du_{\varepsilon,j}\right]\,dx \nonumber\\
&=\int_{\Omega_{r(\varepsilon)}}\left\langle(|Du_{\varepsilon,j}|^2+\delta)^\frac{p-2}{2}Du_{\varepsilon,j}+a_j(x)|Du_{\varepsilon,j}|^{q-2}
Du_{\varepsilon,j},D\varphi\right\rangle\,dx,
\end{align*}
since the $q$-growth term does not have singularity so that it is not necessarily regularized as the $p$-growth term. Then the subsequent processes are analogous to Case 1.

\textbf{Case 3.} $2\le p\leq q$. In this non-singular scenario, we replace the display \eqref{3-5} by
\begin{align*}
&\quad\int_{\Omega_{r(\varepsilon)}}-\varphi\dive\left[|Du_{\varepsilon,j}|^{p-2}Du_{\varepsilon,j}+a_j(x)|Du_{\varepsilon,j}|^{q-2}
Du_{\varepsilon,j}\right]\,dx \nonumber\\
&=\int_{\Omega_{r(\varepsilon)}}\left\langle|Du_{\varepsilon,j}|^{p-2}Du_{\varepsilon,j}+a_j(x)|Du_{\varepsilon,j}|^{q-2}
Du_{\varepsilon,j},D\varphi\right\rangle\,dx.
\end{align*}
The other procedures are also similar to Case 1, even more straightforward due to the absence of $\delta$. All in all, we finally deduce the desired result.
\end{proof}

In the previous lemma, in order to get the inequality \eqref{3-9}, $f(x,\tau,0)$ is {\em a priori} assumed to be non-positive. The forthcoming lemma states this hypotheses exactly can be realized when $Du_\varepsilon=0$.

\begin{lemma}
\label{lem3-3}
Let $u$ be a bounded viscosity supersolution to \eqref{main} in $\Omega$. Let also $0\le a(x)\in C^1(\Omega)$ and the function $f(x,\tau,\xi)$ be continuous in all variables. Whenever $Du_\varepsilon(\hat{x})=0$ for some $\hat{x}\in \Omega_{r(\varepsilon)}$, we have
$$ f_\varepsilon(\hat{x},u_\varepsilon(\hat{x}),Du_\varepsilon(\hat{x}))\le 0.
$$
\end{lemma}

\begin{proof}
We have known by \cite[Lemma 4.3]{JJ12} that if $Du_\varepsilon(\hat{x})=0$, then
$$
u_\varepsilon(\hat{x})=u(\hat{x}).
$$
From the definition of inf-convolution,
$$
u(\hat{x})\le u(y)+\frac{|\hat{x}-y|^s}{s\varepsilon^{s-1}} \quad\text{for all } y\in\Omega.
$$
Now introduce an auxiliary function
$$
\phi(y)=u(\hat{x})-\frac{|\hat{x}-y|^s}{s\varepsilon^{s-1}} \quad\text{with } y\in\Omega,
$$
where $s=2$ if $p\ge2$, and $s>\max\left\{\frac{p}{p-1},\frac{q}{q-1}\right\}=\frac{p}{p-1}$ if $1<p<2$. We can apparently see that
$$
\phi\in C^2(\Omega), \ D\phi(\hat{x})=0, \ D\phi(y)\neq0 \quad  \text{for } y\neq \hat{x},
$$
and $\phi$ touches $u$ from below at $\hat{x}$. For the case $s>2$ (i.e., $1<p<2$), we next evaluate some important quantities,
$$
D\phi=\varepsilon^{1-s}|\hat{x}-y|^{s-2}(\hat{x}-y),\ D^2\phi=-\varepsilon^{1-s}|\hat{x}-y|^{s-2}\left(I+(s-2)\frac{\hat{x}-y}{|\hat{x}-y|}\otimes\frac{\hat{x}-y}{|\hat{x}-y|}\right)
$$
and
$$
\mathrm{tr}D^2\phi=-(n+s-2)\varepsilon^{1-s}|\hat{x}-y|^{s-2}, \ \left\langle\frac{D\phi}{|D\phi|},D^2\phi\frac{D\phi}{|D\phi|}\right\rangle=-(s-1)\varepsilon^{1-s}|\hat{x}-y|^{s-2}.
$$
Combining these quantities leads to
\begin{align*}
&\quad\dive(|D\phi|^{p-2}D\phi+a(y)|D\phi|^{p-2}D\phi)\\
&=\Delta_p\phi+a(y)\Delta_q\phi+|D\phi|^{p-2}D\phi\cdot Da(y)\\
&=-(n+(p-1)(s-1)-1)\varepsilon^{(1-s)(p-1)}|\hat{x}-y|^{(s-1)(p-1)-1}\\
&\quad-a(y)(n+(q-1)(s-1)-1)\varepsilon^{(1-s)(q-1)}|\hat{x}-y|^{(s-1)(q-1)-1}\\
&\quad+\varepsilon^{(1-s)(q-1)}|\hat{x}-y|^{(s-1)(q-1)-1}(\hat{x}-y)\cdot Da(y).
\end{align*}
Observe that $s>\max\left\{\frac{p}{p-1},\frac{q}{q-1}\right\}$, which implies that
$$
(s-1)(p-1)-1>0 \quad \textmd{and} \quad  (s-1)(q-1)-1>0.
$$
Thereby,
$$
\lim_{r\rightarrow0}\sup_{y\in B_r(\hat{x})\setminus\{\hat{x}\}}(-\dive(|D\phi|^{p-2}D\phi+a(y)|D\phi|^{p-2}D\phi))=0.
$$
As for $s=2$ (i.e., $p\ge 2$), the previous identity is valid obviously. Because $u$ is a viscosity supersolution of \eqref{main}, there holds that
$$
\lim_{r\rightarrow0}\sup_{y\in B_r(\hat{x})\setminus\{\hat{x}\}}(-\dive(|D\phi|^{p-2}D\phi+a(y)|D\phi|^{p-2}D\phi))\ge
f(\hat{x},u(\hat{x}),D\phi(\hat{x})).
$$
Recalling that $u(\hat{x})=u_\varepsilon(\hat{x})$, $Du_\varepsilon(\hat{x})=0=D\phi(\hat{x})$, we obtain, from the above two displays,
$$
0\ge f(\hat{x},u_\varepsilon(\hat{x}),Du_\varepsilon(\hat{x}))\ge f_\varepsilon(\hat{x},u_\varepsilon(\hat{x}),Du_\varepsilon(\hat{x})),
$$
as desired.
\end{proof}

Next, we show the convergence of inf-convolution $u_\varepsilon$ in the Musielak-Orlicz-Sobolev space $W^{1,H(\cdot)}(\Omega)$ in the following two lemmas.

\begin{lemma}
\label{lem3-4}
Under \eqref{0-2} and the preconditions of Lemma \ref{lem3-1}, we infer that the function $u\in W^{1,H(\cdot)}_{\rm loc}(\Omega)$ and, up to a subsequence, $Du_\varepsilon\rightarrow Du$ weakly in $L^{H(\cdot)}(\Omega')$ for every $\Omega'\subset\subset\Omega$.
\end{lemma}

\begin{proof}
Take a cut-off function $\psi\in C^\infty_0(\Omega)$ fulfilling $\psi\equiv1$ in $\Omega'$ and $0\le\psi\le1$ in $\Omega$. Set $\varepsilon$ so small that $\mathrm{supp}\,\psi=:E\subset\Omega_{r(\varepsilon)}$. Notice that $u_\varepsilon\rightarrow u$ locally uniformly in $\Omega$, so we can define a test function
$$
\varphi:=(K-u_\varepsilon)\psi^q(\ge0)
$$
with $K:=\sup_{\varepsilon;\, x\in \Omega'}|u_\varepsilon(x)|$ (finite). Hence, utilizing Lemma \ref{lem3-2},
\begin{align*}
&\quad \int_{\Omega_{r(\varepsilon)}}\varphi f_\varepsilon(x,u_\varepsilon,Du_\varepsilon)\,dx+E(\varepsilon)\int_{\Omega_{r(\varepsilon)}\setminus\{Du_\varepsilon=0\}}\varphi\,dx\\
&\leq
\int_{\Omega_{r(\varepsilon)}}\langle |Du_\varepsilon|^{p-2}Du_\varepsilon+a(x)|Du_\varepsilon|^{q-2}Du_\varepsilon,D\varphi\rangle\,dx\\
&=\int_{\Omega_{r(\varepsilon)}}q\psi^{q-1}(K-u_\varepsilon)\langle |Du_\varepsilon|^{p-2}Du_\varepsilon+a(x)|Du_\varepsilon|^{q-2}Du_\varepsilon,D\psi\rangle\,dx\\
&\quad-\int_{\Omega_{r(\varepsilon)}}\psi^q H(x,Du_\varepsilon)\,dx.
\end{align*}
Namely,
\begin{align*}
\int_{\Omega_{r(\varepsilon)}}\psi^q H(x,Du_\varepsilon)\,dx&\leq q\int_{\Omega_{r(\varepsilon)}}\psi^{q-1}(K-u_\varepsilon) (|Du_\varepsilon|^{p-1}+a(x)|Du_\varepsilon|^{q-1})|D\psi|\,dx\\
&\quad+\int_{\Omega_{r(\varepsilon)}}\varphi |f_\varepsilon(x,u_\varepsilon,Du_\varepsilon)|\,dx+|E(\varepsilon)|\int_{\Omega_{r(\varepsilon)}}\varphi\,dx\\
&=:I_1+I_2+I_3.
\end{align*}
Noting $q-1\ge \frac{q(p-1)}{p}$, $0\le \psi\le1$ and applying Young's inequality with $\epsilon$, we have
\begin{equation*}
\begin{split}
I_1&\le q\int_{\Omega_{r(\varepsilon)}}(K-u_\varepsilon)|D\psi|\psi^\frac{q(p-1)}{p}|Du_\varepsilon|^{p-1}+a(x)(K-u_\varepsilon)|D\psi|\psi^{q-1}|Du_\varepsilon|^{q-1}\,dx\\
&\le \epsilon\int_{\Omega_{r(\varepsilon)}}\psi^qH(x,Du_\varepsilon)\,dx+C(q,\epsilon)\int_{\Omega_{r(\varepsilon)}}H(x,(K-u_\varepsilon)|D\psi|)\,dx.
\end{split}
\end{equation*}
In view of the growth condition on $f$, we can deal with $I_2$ as
\begin{align*}
I_2&\le \int_{\Omega_{r(\varepsilon)}}(K-u_\varepsilon)\psi^q \gamma_\infty(|Du_\varepsilon|^{p-1}+a(x)|Du_\varepsilon|^{q-1})+(K-u_\varepsilon)\psi^q\Phi\,dx\\
&\le\epsilon\int_{\Omega_{r(\varepsilon)}}\psi^qH(x,Du_\varepsilon)\,dx+C(\gamma_\infty,\epsilon)\int_{\Omega_{r(\varepsilon)}}\psi^qH(x,K-u_\varepsilon)\,dx
+C(K,\|\Phi\|_{L^\infty(E)},E).
\end{align*}
Here $\gamma_\infty:=\max_{\tau\in[0,K]}\gamma(\tau)$. For $I_3$, we have
$$
I_3\le C(K,E),
$$
where we assume $|E(\varepsilon)|\le 1$ without loss of generality. Choosing proper $\epsilon\in(0,1)$ and merging these above estimates yields that
$$
\int_{\Omega'}H(x,Du_\varepsilon)\,dx\le \int_{\Omega_{r(\varepsilon)}}\psi^qH(x,Du_\varepsilon)\,dx\le C(p,q,a,K,\gamma,\Phi,D\psi,E).
$$
This indicates that $Du_\varepsilon$ is uniformly bounded in $L^{H(\cdot)}(\Omega')$ with respect to $\varepsilon$, which further deduces that there exists a function $Du\in L^{H(\cdot)}(\Omega')$ such that $Du_\varepsilon\rightarrow Du$ weakly in $L^{H(\cdot)}(\Omega')$ up to a subsequence owing to $L^{H(\cdot)}(\Omega')$ being a reflexible Banach space. Finally, $u$ belongs to $W^{1,H(\cdot)}(\Omega')$ with $Du$ as its weak derivative.
\end{proof}

\begin{lemma}
\label{lem3-5}
With \eqref{0-2} and the hypotheses of Lemma \ref{lem3-1}, we arrive at $u_\varepsilon\rightarrow u$ as $\varepsilon\rightarrow0$, up to a subsequence, in $W^{1,H(\cdot)}(\Omega')$ for each $\Omega'\subset\subset\Omega$.
\end{lemma}

\begin{proof}
Take a cut-off function $\psi\in C^\infty_0(\Omega)$ satisfying $\psi\equiv1$ in $\Omega'$ and $0\le\psi\le1$ in $\Omega$. Let $\varepsilon$ so small that $\mathrm{supp}\,\psi=:E\subset\Omega_{r(\varepsilon)}$. Notice that $u_\varepsilon\le u$ in $\Omega$, so we define a test function
$$
\varphi:=(u-u_\varepsilon)\psi(\ge0).
$$
It is easy to find that $\varphi\in W^{1,H(\cdot)}_0(\Omega_{r(\varepsilon)})$. Employing again Lemma \ref{lem3-2} obtains
\begin{align*}
&\quad\int_{\Omega_{r(\varepsilon)}}\langle A(x,Du)-A(x,Du_\varepsilon),D\varphi\rangle\,dx\\
&\le \int_{\Omega_{r(\varepsilon)}}\langle A(x,Du),D\varphi\rangle\,dx+\int_{\Omega_{r(\varepsilon)}}\varphi|f_\varepsilon(x,u_\varepsilon,Du_\varepsilon)|\,dx+|E(\varepsilon)|\int_{\Omega_{r(\varepsilon)}}\varphi\,dx.
\end{align*}
After manipulation, we get
\begin{align*}
&\quad\int_{\Omega_{r(\varepsilon)}}\langle A(x,Du)-A(x,Du_\varepsilon),Du-Du_\varepsilon\rangle\psi\,dx\\
&\le \int_E\psi\langle A(x,Du),Du-Du_\varepsilon\rangle\,dx+\int_E(u-u_\varepsilon)\langle A(x,Du_\varepsilon),D\psi\rangle\,dx\\
&\quad+\int_E\psi(u-u_\varepsilon)|f_\varepsilon(x,u_\varepsilon,Du_\varepsilon)|\,dx+|E(\varepsilon)|\int_E(u-u_\varepsilon)\psi\,dx\\
&=J_1+J_2+J_3+J_4.
\end{align*}
For $J_2$, via the inequalities \eqref{2-0-2} and \eqref{2-0-3}, we derive
\begin{align*}
J_2&\le \|u-u_\varepsilon\|_{L^\infty(E)}\int_E|A(x,Du_\varepsilon)||D\psi|\,dx\\
&\le C\|u-u_\varepsilon\|_{L^\infty(E)}\int_EH^*(x,|A(x,Du_\varepsilon)|)+H(x,|D\psi|)\,dx\\
&\le C\|u-u_\varepsilon\|_{L^\infty(E)}\int_EH(x,|Du_\varepsilon|)+H(x,|D\psi|)\,dx.
\end{align*}
Using the growth assumption on $f$ and Young's inequality leads to
\begin{align*}
J_3&\le \|u-u_\varepsilon\|_{L^\infty(E)}\int_E\gamma_\infty(|Du_\varepsilon|^{p-1}+a(x)|Du_\varepsilon|^{q-1})+\Phi\,dx\\
&\le C\|u-u_\varepsilon\|_{L^\infty(E)}\left[\int_EH(x,|Du_\varepsilon|)\,dx+(1+\|\Phi\|_{L^\infty(E)})|E|\right],
\end{align*}
where $\gamma_\infty$ is the same as that in Lemma \ref{lem3-4}.
As for $J_4$,
$$
J_4\le|E(\varepsilon)|\|u-u_\varepsilon\|_{L^\infty(E)}|E|.
$$
Recalling that $Du_\varepsilon\rightarrow Du$ weakly in $L^{H(\cdot)}(E)$ in Lemma \ref{lem3-4} and $u_\varepsilon\rightarrow u$ locally uniformly in Proposition \ref{pro2-1}, we know $J_1+J_2+J_3+J_4$ tends to 0. In other words,
$$
\lim_{\varepsilon\rightarrow0}\int_{\Omega'}\langle A(x,Du)-A(x,Du_\varepsilon),Du-Du_\varepsilon\rangle\,dx=0.
$$
Following the calculations in \cite[page 9]{FZ2020}, we further arrive at
$$
\lim_{\varepsilon\rightarrow0}\int_{\Omega'}H(x,Du-Du_\varepsilon)\,dx=0,
$$
which implies the desired result.
\end{proof}

Finally, we end this section by verifying that the locally Lipschitz continuous viscosity supersolutions to \eqref{main} are also weak supersolutions, as stated in Theorem \ref{thm1}. We next intend to apply the previous convergence results to pass to the limit in the display of Lemma \ref{lem3-2}.

\bigskip

\noindent\textbf{Proof of Theorem \ref{thm1}.}
Let $\varphi\in C^\infty_0(\Omega)$ be a nonnegative test function and set an open $\Omega'\subset\subset\Omega$ such that $\mathrm{supp}\,\varphi\subset\Omega'$. Fix a sufficiently small $\varepsilon_0>0$ fulfilling $\Omega'\subset\Omega_{r(\varepsilon)}$ for $0<\varepsilon<\varepsilon_0$. Now we want to show
\begin{equation}
\label{3-10}
\int_\Omega\langle |Du|^{p-2}Du+a(x)|Du|^{q-2}Du,D\varphi\rangle\,dx\ge\int_\Omega\varphi f(x,u,Du)\,dx,
\end{equation}
which is the definition of weak supersolution of \eqref{main}. This desired claim shall follow through Lemma \ref{lem3-2}, once the forthcoming displays are justified:
\begin{equation}
\label{3-11}
\begin{split}
&\quad\lim_{\varepsilon\rightarrow0}\int_{\Omega'}\langle |Du_\varepsilon|^{p-2}Du_\varepsilon+a(x)|Du_\varepsilon|^{q-2}Du_\varepsilon,D\varphi\rangle\,dx\\
&=\int_{\Omega'}\langle |Du|^{p-2}Du+a(x)|Du|^{q-2}Du,D\varphi\rangle\,dx,
\end{split}
\end{equation}
\begin{equation}
\label{3-12}
\lim_{\varepsilon\rightarrow0}\int_{\Omega'}\varphi f_\varepsilon(x,u_\varepsilon,Du_\varepsilon)\,dx=\int_{\Omega'}\varphi f(x,u,Du)\,dx
\end{equation}
and
\begin{equation}
\label{3-13}
\lim_{\varepsilon\rightarrow0}E(\varepsilon)\int_{\Omega'}\varphi\,dx=0.
\end{equation}
First, the limit \eqref{3-13} is obviously valid. Next, we demonstrate the validity of \eqref{3-11}. We shall employ the elementary vector inequality (see \cite{Lin17}):
\begin{equation}
\label{3-14}
\big||\xi_1|^{t-2}\xi_1-|\xi_2|^{t-2}\xi_2\big|\leq\begin{cases}(t-1)|\xi_1-\xi_2|(|\xi_1|^{t-2}+|\xi_2|^{t-2}), \quad &\textmd{if }{t\geq2}, \\[2mm]
2^{2-t}|\xi_1-\xi_2|^{t-1}, \quad &\textmd{if }{1<t<2},\end{cases}
\end{equation}
where $\xi_1,\xi_2\in \mathbb{R}^n$. We split the proof of \eqref{3-11} into three cases.

For the case $1<p\le q<2$, we use \eqref{2-0-2}, \eqref{2-1}--\eqref{2-1-2}, the basic inequality \eqref{3-14} to get
\begin{align*}
&\quad\int_{\Omega'}\langle |Du_\varepsilon|^{p-2}Du_\varepsilon-|Du|^{p-2}Du+a(x)(|Du_\varepsilon|^{q-2}Du_\varepsilon-|Du|^{q-2}Du),D\varphi\rangle\,dx\\
&\le\int_{\Omega'}(||Du_\varepsilon|^{p-2}Du_\varepsilon-|Du|^{p-2}Du|+a(x)||Du_\varepsilon|^{q-2}Du_\varepsilon-|Du|^{q-2}Du|)|D\varphi|\,dx\\
&\le C\int_{\Omega'}(|Du_\varepsilon-Du|^{p-1}+a(x)|Du_\varepsilon-Du|^{q-1})|D\varphi|\,dx\\
&\le C\|h(x,Du_\varepsilon-Du)\|_{L^{H^*(\cdot)}(\Omega')}\|D\varphi\|_{L^{H(\cdot)}(\Omega')}\\
&\le C\|D\varphi\|_{L^{H(\cdot)}(\Omega')}\max\left\{\left(\varrho_{H^*}(h(x,Du_\varepsilon-Du))\right)^\frac{q}{q+1},\left(\varrho_{H^*}(h(x,Du_\varepsilon-Du))\right)^\frac{p}{p+1}\right\}\\
&\le C\|D\varphi\|_{L^{H(\cdot)}(\Omega')}\max\left\{\left(\int_{\Omega'}H(x,Du_\varepsilon-Du)\,dx\right)^\frac{q}{q+1},\left(\int_{\Omega'}H(x,Du_\varepsilon-Du)\,dx\right)^\frac{p}{p+1}\right\}\\
&\rightarrow0
\end{align*}
as $\varepsilon\rightarrow0$, by applying $Du_\varepsilon\rightarrow Du$ in $L^{H(\cdot)}(\Omega')$ in Lemma \ref{lem3-5}. Here $h(x,z):=|z|^{p-1}+a(x)|z|^{q-1}$.

When $2\le p\le q$, exploiting again \eqref{2-0-2}, \eqref{2-1}--\eqref{2-1-2} as well as \eqref{3-14}, and applying $Du_\varepsilon\rightarrow Du$ in $L^{H(\cdot)}(\Omega')$, we can see that
\begin{align*}
&\quad\int_{\Omega'}\langle |Du_\varepsilon|^{p-2}Du_\varepsilon-|Du|^{p-2}Du+a(x)(|Du_\varepsilon|^{q-2}Du_\varepsilon-|Du|^{q-2}Du),D\varphi\rangle\,dx\\
&\le C\int_{\Omega'}\left[(|Du_\varepsilon|^{p-2}+|Du|^{p-2})+a(x)(|Du_\varepsilon|^{q-2}+|Du|^{q-2})\right]|Du_\varepsilon-Du||D\varphi|\,dx\\
&\le C\|D\varphi\|_{L^\infty(\Omega')}\int_{\Omega'}\left[(1+|Du_\varepsilon|^{p-1}+|Du|^{p-1})+a(x)(1+|Du_\varepsilon|^{q-1}+|Du|^{q-1})\right]\\
&\qquad\qquad\qquad\qquad\qquad\times|Du_\varepsilon-Du|\,dx\\
&=C\|D\varphi\|_{L^\infty(\Omega')}\int_{\Omega'}[(1+a(x))+h(x,Du_\varepsilon)+h(x,Du)]|Du_\varepsilon-Du|\,dx\\
&\le C\|Du_\varepsilon-Du\|_{L^{H(\cdot)}(\Omega')}+C\|h(x,Du)\|_{L^{H^*(\cdot)}(\Omega')}\|Du_\varepsilon-Du\|_{L^{H(\cdot)}(\Omega')}\\
&\quad+C\|h(x,Du_\varepsilon)\|_{L^{H^*(\cdot)}(\Omega')}\|Du_\varepsilon-Du\|_{L^{H(\cdot)}(\Omega')}\\
&\rightarrow0
\end{align*}
as $\varepsilon\rightarrow0$, where the constant $C$ depends on $p,q,\|a\|_{L^\infty(\Omega')},\|D\varphi\|_{L^\infty(\Omega')},|\Omega|$. Here we need to notice that the quantity $\int_{\Omega'}H(x,Du_\varepsilon)\,dx$ is uniformly bounded.

In the last case $1<p<2\le q$, we combine the previous two scenarios to deduce the claim \eqref{3-11}. Specifically,
\begin{align*}
&\quad\int_{\Omega'}\langle |Du_\varepsilon|^{p-2}Du_\varepsilon-|Du|^{p-2}Du+a(x)(|Du_\varepsilon|^{q-2}Du_\varepsilon-|Du|^{q-2}Du),D\varphi\rangle\,dx\\
&\le C\int_{\Omega'}\left[|Du_\varepsilon-Du|^{p-1}+a(x)(|Du_\varepsilon|^{q-2}+|Du|^{q-2})|Du_\varepsilon-Du|\right]|D\varphi|\,dx\\
&\le C\int_{\Omega'}(|Du_\varepsilon-Du|^{p-1}+a(x)|Du_\varepsilon-Du|^{q-1})|D\varphi|\,dx\\
&\quad+C\int_{\Omega'} a(x)(1+|Du_\varepsilon|^{q-1}+|Du|^{q-1})|Du_\varepsilon-Du||D\varphi|\,dx\\
&\le C\int_{\Omega'}h(x,Du_\varepsilon-Du)|D\varphi|\,dx+C\int_{\Omega'}|Du_\varepsilon-Du|\,dx\\
&\quad+ C\int_{\Omega'}h(x,Du_\varepsilon)|Du_\varepsilon-Du|\,dx+C\int_{\Omega'}h(x,Du)|Du_\varepsilon-Du|\,dx,
\end{align*}
which tends to 0 as $\varepsilon\rightarrow0$, where $C$ depends upon $p,q,\|a\|_{L^\infty(\Omega')},\|D\varphi\|_{L^\infty(\Omega')},|\Omega'|$.

Eventually, let us prove the claim \eqref{3-12}. Via the uniform continuity of $f$, for each $\epsilon>0$, there is a $\delta>0$ depending only on $\epsilon$ such that
$$
|f(x,u_\varepsilon(x),Du_\varepsilon(x))-f(y,u_\varepsilon(x),Du_\varepsilon(x))|\le\epsilon \quad \text{for } y\in B_\delta(x).
$$
Now select $\varepsilon'_0>0$ to satisfy $r(\varepsilon)<\delta$ if $0<\varepsilon<\varepsilon'_0$. Then we have
$$
f(x,u_\varepsilon(x),Du_\varepsilon(x))<\epsilon+f(y,u_\varepsilon(x),Du_\varepsilon(x))
$$
for each $x\in \Omega'$ and $y\in B_{r(\varepsilon)}(x)$. In particular,
$$
f(x,u_\varepsilon(x),Du_\varepsilon(x))<\epsilon+f_\varepsilon(x,u_\varepsilon(x),Du_\varepsilon(x))
$$
by the definition of $f_\varepsilon$, and moreover
$$
0\le |f(x,u_\varepsilon(x),Du_\varepsilon(x))-f_\varepsilon(x,u_\varepsilon(x),Du_\varepsilon(x))|\le\epsilon.
$$
We thus have
$$
\int_{\Omega'}|f(x,u_\varepsilon,Du_\varepsilon)-f_\varepsilon(x,u_\varepsilon,Du_\varepsilon)|\varphi\,dx\leq\epsilon\|\varphi\|_{L^\infty(\Omega')}|\Omega|.
$$
In view of (1) in Proposition \ref{pro2-1}, it is known that $\|u_\varepsilon\|_{L^\infty(\Omega')}\le\|u\|_{L^\infty(\Omega')}$ for any $\varepsilon$. Namely,
$$
\max_{[0,\|u_\varepsilon\|_{L^\infty(\Omega')}]}\{\gamma(\tau)\}\le\max_{[0,\|u\|_{L^\infty(\Omega')}]}\{\gamma(\tau)\}=:\gamma_\infty.
$$
According to the growth condition on $f$,
$$
|f(x,u_\varepsilon,Du)|\le\gamma_\infty(|Du|^{p-1}+a(x)|Du|^{q-1})+\Phi(x) \quad\text{in } \Omega',
$$
which belongs to $L^{H^*(\cdot)}(\Omega')$. Thereby we can exploit the dominated convergence theorem to infer
$$
\lim_{\varepsilon\rightarrow0}\int_{\Omega'}|f(x,u_\varepsilon,Du)-f(x,u,Du)|\varphi\,dx=0.
$$
We finally address the term $\int_{\Omega'}|f(x,u_\varepsilon,Du_\varepsilon)-f(x,u_\varepsilon,Du)|\varphi\,dx$,
\begin{align*}
&\quad\int_{\Omega'}|f(x,u_\varepsilon,Du_\varepsilon)-f(x,u_\varepsilon,Du)|\varphi\,dx\\
&\le C\int_{\Omega'}|Du_\varepsilon-Du|\varphi\,dx\\
&\le C\|\varphi\|_{L^{H^*(\cdot)}(\Omega')}\|Du_\varepsilon-Du\|_{L^{H(\cdot)}}(\Omega')\\
&\rightarrow0 \quad \text{as } \varepsilon\rightarrow0,
\end{align*}
where we have utilized the Lipschitz continuity of $f$ in the third variable, and the convergence $Du_\varepsilon\rightarrow Du$ in $L^{H(\cdot)}(\Omega')$.

Merging these above estimates yields that
\begin{align*}
&\quad\int_{\Omega'}|f_\varepsilon(x,u_\varepsilon,Du_\varepsilon)-f(x,u,Du)|\varphi\,dx\\
&\le \int_{\Omega'}|f_\varepsilon(x,u_\varepsilon,Du_\varepsilon)-f(x,u_\varepsilon,Du_\varepsilon)|\varphi\,dx+\int_{\Omega'}|f(x,u_\varepsilon,Du_\varepsilon)-f(x,u_\varepsilon,Du)|\varphi\,dx\\
&\quad+\int_{\Omega'}|f(x,u_\varepsilon,Du)-f(x,u,Du)|\varphi\,dx
\end{align*}
converges to 0 by sending $\varepsilon\rightarrow0$. Hereto, we have verified the displays \eqref{3-11}--\eqref{3-13}, from which we get the inequality \eqref{3-10}.

\section{weak solutions are viscosity solutions}
\label{sec-4}

In this section, we prove that weak supersolutions are viscosity supersolutions to \eqref{main}, that is Theorem \ref{thm2}, by using comparison principle for weak solutions, subsequently giving two examples of comparison results. Then weak subsolutions can be showed to be viscosity subsolutions in a similar way.

\bigskip

\noindent\textbf{Proof of Theorem \ref{thm2}.}
We argue by contradiction. If not, there exists a $\varphi\in C^2(\Omega)$ touching $u$ from below at $x_0\in\Omega$, that is,
\begin{equation*}
\begin{cases}
u(x_0)=\varphi(x_0),\\
u(x)>\varphi(x) \quad\textmd{for }  x\neq x_0,\\
D\varphi(x_0)\neq 0,
\end{cases}
\end{equation*}
and however,
$$
-\dive A(x_0,D\varphi(x_0))<f(x_0,u(x_0),D\varphi(x_0)).
$$
By means of continuity, for some $\delta>0$ there is a small enough $r>0$ such that
\begin{equation}
\label{4-1}
-\dive A(x,D\varphi(x))\le f(x,\varphi,D\varphi)-\delta
\end{equation}
for $x\in B_r:=B_r(x_0)$. Denote $\widetilde{\varphi}:=\varphi+m$ with $m>0$ to be chosen later. In view of \eqref{4-1}, we have
\begin{align*}
&\quad-\dive A(x,D\widetilde{\varphi})-f(x,\widetilde{\varphi},D\widetilde{\varphi})\\
&=-\dive A(x,D\varphi)-f(x,\varphi,D\varphi)+f(x,\varphi,D\varphi)-f(x,\widetilde{\varphi},D\widetilde{\varphi})\\
&\le -\delta+f(x,\varphi,D\varphi)-f(x,\varphi+m,D\varphi)\\
&\le -\frac{\delta}{2},
\end{align*}
if $m>0$ is sufficiently small. Indeed, we can pick $m\in \left(0,\frac{1}{2}\min_{\partial B_r}\{u-\varphi\}\right)$ (note the lower semicontinuity of $u$) small, taking into account the uniform continuity of $f$, such that
$$
|f(x,\varphi,D\varphi)-f(x,\varphi+m,D\varphi)|\le \frac{\delta}{2}.
$$
Hence, $\widetilde{\varphi}$ is a weak subsolution to \eqref{main} in $B_r$ as well. Observe that $\widetilde{\varphi}=\varphi+m<\varphi+u-\varphi=u$ on $\partial B_r$. Thereby through the (CPP) we get $u\ge \widetilde{\varphi}$ in $B_r$. Nonetheless, $u(x_0)=\varphi(x_0)<\varphi(x_0)+m$, which is a contradiction. Then $u$ is a viscosity supersolution to \eqref{main}.

\medskip

A fundamental issue in Theorem \ref{thm2} is the availability of comparison principle for weak solutions. Nevertheless, it is not undemanding to establish such principle for the nonhomogeneous double phase equations with very general structure. Hence, we next for two slightly special cases prove comparison principle for weak solutions.

\begin{lemma}
\label{lem4-1}
Assume that $u,v$ are the weak subsolution and supersolution, respectively, to $-\dive A(x,Dw)=f(x,w)$ in $\Omega$, where $f$ is decreasing in the $w$-variable. If $u\le v$ on $\partial \Omega$, then $u\le v$ a.e. in $\Omega$.
\end{lemma}

\begin{proof}
Since $u,v$ are separately weak subsolution and supersolution such that $u\le v$ on $\partial\Omega$, the auxiliary function
$$
w:=(u-v-l)_+, \ l>0,
$$
can be chosen as a test function, which belongs to the space $W^{1,H(\cdot)}_0(\Omega)$. Therefore, we get
$$
\int_\Omega\langle A(x,Du)-A(x,Dv),Dw\rangle\,dx\le\int_\Omega(f(x,u)-f(x,v))w\,dx.
$$
We use the fact that $f$ is decreasing with respect to the second variable to arrive at
$$
\int_\Omega(f(x,u)-f(x,v))(u-v-l)_+\,dx\le0.
$$
Furthermore, due to the strictly monotone increasing property of the operator $A(x,\cdot)$, i.e., $\langle A(x,\xi)-A(x,\zeta),\xi-\zeta\rangle>0$ for $\xi\neq\zeta\in \mathbb{R}^n$, it follows that
$$
0\le \int_{\Omega\cap\{(u-v-l)_+>0\}}\langle A(x,Du)-A(x,Dv),Du-Dv\rangle\,dx\le0.
$$
This implies that $w\equiv0$ a.e. in $\Omega$. In other words, $u\le v+l$ a.e. in $\Omega$. Letting $l\rightarrow0$, we can see that $u\le v$ a.e. in $\Omega$.
\end{proof}

\begin{lemma}
\label{lem4-2}
Suppose that $f(x,\tau,\eta)$ is decreasing in $\tau$, and is locally Lipschitz continuous with respect to $\eta$ in $\Omega\times\mathbb{R}\times\mathbb{R}^n$. Let $u,v$ be the weak subsolution and supersolution respectively to \eqref{main} such that
$$
|Du(x)|+|Dv(x)|\ge \delta \quad \text{a.e.} \ x\in\Omega
$$
with $\delta>0$ any number. Let also $2\le p\le q$. Then there is $\varepsilon>0$ such that, for every domain $E\subset\subset\Omega$ fulfilling $|E|\le\varepsilon$, whenever $u\le v$ on $\partial E$ then it holds that $u\le v$ a.e. in $E$.
\end{lemma}

\begin{proof}
Selecting $w=(u-v)_+\chi_E\in W^{1,H(\cdot)}_0(\Omega)$ as a test function in the weak formulation of \eqref{main}, we derive, from the hypotheses on $f$,
\begin{align}
\label{4-2}
&\quad\int_\Omega\langle A(x,Du)-A(x,Dv),Dw\rangle\,dx \nonumber\\
&\le\int_\Omega(f(x,u,Du)-f(x,v,Dv))w\,dx \nonumber\\
&=\int_\Omega(f(x,u,Du)-f(x,v,Du))(u-v)_+\chi_E+(f(x,v,Du)-f(x,v,Dv))(u-v)_+\chi_E\,dx  \nonumber\\
&\le \int_\Omega(f(x,v,Du)-f(x,v,Dv))(u-v)_+\chi_E\,dx  \nonumber\\
&\le C\int_E|Du-Dv|(u-v)_+\,dx.
\end{align}
Now by the basic inequality
$$
C(|\xi|+|\eta|)^{p-2}|\xi-\eta|^2\le (|\xi|^{p-2}\xi-|\eta|^{p-2}\eta)\cdot(\xi-\eta) \quad\text{for } p\ge2,
$$
\eqref{4-2} turns into
\begin{align*}
&\quad\int_E(|Du|+|Dv|)^{p-2}|D(u-v)_+|^2+a(x)(|Du|+|Dv|)^{q-2}|D(u-v)_+|^2\,dx\\
&\leq C\int_E|Du-Dv|(u-v)_+\,dx\\
&\le C(E)\int_E|D(u-v)_+|^2\,dx\\
&\le C(E)\int_E(|Du|+|Dv|)^{2-p}(|Du|+|Dv|)^{p-2}|D(u-v)_+|^2\,dx\\
&\le \delta^{2-p}C(E)\int_E(|Du|+|Dv|)^{p-2}|D(u-v)_+|^2\,dx,
\end{align*}
where in the third line we utilized the H\"{o}lder and Poincar\'{e} inequalities, and the quantity $C(E)$ will tend to 0 when the measure $|E|$ goes to 0. From above, we can find that if the domain $E$ is small enough, the last inequality is self-contradictory. That is, $(u-v)_+\equiv0$ in $E$, which leads to $u\le v$ a.e. in $E$.
\end{proof}

\begin{remark}
Lemma \ref{lem4-1} could be exploited in the proof of Theorem \ref{thm2} apparently. Observe that $D\varphi(x_0)\neq0$ and $\varphi\in C^2(\Omega)$, so we can take such small ball $B_r(x_0)$ that $|B_r(x_0)|\le \varepsilon$ ($\varepsilon$ is provided by Lemma \ref{lem4-2}) and moreover $|D\varphi(x)|\ge\delta>0$ in $B_r(x_0)$. Thus we can keep track of the  proof of Theorem \ref{thm2} to deduce that a lower semicontinuous weak supersolution is  a viscosity supersolution to \eqref{main} by  Lemma \ref{lem4-2}.
\end{remark}

\section{Lipschitz continuity of viscosity solutions}
\label{sec-5}

We in this part show that the bounded viscosity solutions to \eqref{main} are locally Lipschitz continuous. The strategy is to verify first the H\"{o}lder continuity of viscosity solutions by using the Ishii-Lions methods, and further, based on the H\"{o}lder continuity, to demonstrate the Lipschitz continuity of viscosity solutions through the Ishii-Lions methods again. The similar idea can be found for instance in \cite{AR18}. For the sake of convenience, we suppose the domain $\Omega$ is a unit ball $B_1$.

\begin{lemma}[Local H\"{o}lder continuity]
\label{lem5-1}
Let $u$ be a bounded viscosity solution to \eqref{main} in $B_1$. Assume that $0\le a(x)\in C^1(B_1)$, $p\le q\le p+1$ and \eqref{0-2} are in force. Then for each $\beta\in(0,1)$, there exists a constant $C>0$, depending on $n,p,q,\beta,\gamma_\infty,\|a\|_{C^1(B_1)},\|u\|_{L^\infty(B_1)}$ and $\|\Phi\|_{L^\infty(B_1)}$, such that
$$
|u(x)-u(y)|\le C|x-y|^\beta
$$
for any $x,y\in B_{3/4}$, where $\gamma_\infty:=\max_{t\in[0,\|u\|_{L^\infty(B_1)}]} \gamma(t)$.
\end{lemma}

\begin{proof}
Fix $x_0,y_0\in B_{3/4}$. We now aim at showing that there are two proper constants $L_1,L_2>0$ such that
\begin{equation}
\label{5-1}
\omega:=\sup_{x,y\in \overline{B_{3/4}}} \left(u(x)-u(y)-L_1\phi(|x-y|)-\frac{L_2}{2}|x-x_0|^2-\frac{L_2}{2}|y-y_0|^2\right)\le0,
\end{equation}
where $\phi(r)=r^\beta$ with $\beta\in(0,1)$.

To this end, assume on the contrary that \eqref{5-1} is not true and moreover $(\overline{x},\overline{y})\in\overline{B_{3/4}}\times\overline{B_{3/4}}$ stands for the point where the supremum is achieved. We can easily know two facts that $\overline{x}\neq\overline{y}$ by $\omega>0$, and $\overline{x},\overline{y}\in B_{3/4}$ by choosing
$$
L_2\ge \frac{64\|u\|_{L^\infty(B_1)}}{(\min\{\mathrm{dist}(x_0,\partial B_{3/4}),\mathrm{dist}(y_0,\partial B_{3/4})\})^2}.\
$$
Besides,
$$
|\overline{x}-\overline{y}|\leq\left(\frac{2\|u\|_{L^\infty(B_1)}}{L_1}\right)^\frac{1}{\beta}
$$
is small enough, provided that $L_1$ is sufficiently large, which shall be utilized later.

Now by invoking theorem of sums \cite[Theorem 3.2]{CIL92}, there are $X,Y\in \mathcal{S}^n$ such that
$$
(\eta_1,X+L_2I)\in \overline{J}^{2,+}u(\overline{x}) \quad \text{and}\quad (\eta_2,Y-L_2I)\in \overline{J}^{2,-}u(\overline{y})
$$
with
\begin{align*}
&\eta_1=L_1D_x\phi(|\overline{x}-\overline{y}|)+L_2(\overline{x}-x_0)=L_1\phi'(|\overline{x}-\overline{y}|)\frac{\overline{x}-\overline{y}}{|\overline{x}-\overline{y}|}+L_2(\overline{x}-x_0),\\
&\eta_2=-L_1D_y\phi(|\overline{x}-\overline{y}|)-L_2(\overline{y}-y_0)=L_1\phi'(|\overline{x}-\overline{y}|)\frac{\overline{x}-\overline{y}}{|\overline{x}-\overline{y}|}-L_2(\overline{y}-y_0).
\end{align*}
Via selecting $L_1\ge C(\beta)L_2$ large enough, there holds that
\begin{equation}
\label{5-2}
\frac{\beta L_1}{2}|\overline{x}-\overline{y}|^{\beta-1}\leq |\eta_1|,|\eta_2|\leq 2\beta L_1|\overline{x}-\overline{y}|^{\beta-1}.
\end{equation}
Furthermore, applying \cite[Theorem 12.2]{Cra97}, for any $\tau>0$ such that $\tau Z<I$, we obtain
\begin{equation}
\label{5-3}
-\frac{2}{\tau}\left(\begin{array}{cc}
I & \\
 &I
\end{array}
\right)\leq\left(\begin{array}{cc}
X&\\
&-Y
\end{array}
\right)\leq
\left(\begin{array}{cc}
Z^\tau&-Z^\tau\\
-Z^\tau&Z^\tau
\end{array}
\right),
\end{equation}
where
\begin{align*}
Z&=L_1\phi''(|\overline{x}-\overline{y}|)\frac{\overline{x}-\overline{y}}{|\overline{x}-\overline{y}|}
\otimes\frac{\overline{x}-\overline{y}}{|\overline{x}-\overline{y}|}+\frac{L_1\phi'(|\overline{x}-\overline{y}|)}{|\overline{x}-\overline{y}|}
\left(I-\frac{\overline{x}-\overline{y}}{|\overline{x}-\overline{y}|}
\otimes\frac{\overline{x}-\overline{y}}{|\overline{x}-\overline{y}|}\right)\\
&=\beta L_1|\overline{x}-\overline{y}|^{\beta-2}\left(I+(\beta-2)\frac{\overline{x}-\overline{y}}{|\overline{x}-\overline{y}|}
\otimes\frac{\overline{x}-\overline{y}}{|\overline{x}-\overline{y}|}\right)
\end{align*}
and
$$
Z^\tau=(I-\tau Z)^{-1}Z
$$
with $(I-\tau Z)^{-1}$ denoting the inverse of the matrix $I-\tau Z$. Now pick $\tau=\frac{1}{2\beta L_1|\overline{x}-\overline{y}|^{\beta-2}}$ such that
$$
Z^\tau=2\beta L_1|\overline{x}-\overline{y}|^{\beta-2}\left(I-2\frac{2-\beta}{3-\beta}\frac{\overline{x}-\overline{y}}{|\overline{x}-\overline{y}|}
\otimes\frac{\overline{x}-\overline{y}}{|\overline{x}-\overline{y}|}\right).
$$
Observe that
\begin{equation}
\label{5-4}
\langle Z^\tau\xi,\xi\rangle=2\beta\frac{\beta-1}{3-\beta}L_1|\overline{x}-\overline{y}|^{\beta-2}<0
\end{equation}
for $\xi=\frac{\overline{x}-\overline{y}}{|\overline{x}-\overline{y}|}$. In addition, it follows from \eqref{5-3} that $X\le Y$ and
\begin{equation}
\label{5-5}
\|X\|,\|Y\|\le 4\beta L_1|\overline{x}-\overline{y}|^{\beta-2}.
\end{equation}

Let
$$
A_s(\eta):=I+(s-2)\frac{\eta}{|\eta|}\otimes\frac{\eta}{|\eta|}  \quad\text{for } \eta\in\mathbb{R}^n\setminus\{0\}
$$
with $s\in\{p,q\}$. An obvious fact is that the eigenvalues of $A_s(\eta)$ belong to the interval $[\min\{1,s-1\},\max\{1,s-1\}]$. Now since $u$ is a viscosity solution to \eqref{main}, we have
$$
F(\overline{x},\eta_1,X+L_2I)-f(\overline{x},u(\overline{x}),\eta_1)\le0
$$
and
$$
F(\overline{y},\eta_2,Y-L_2I)-f(\overline{y},u(\overline{y}),\eta_2)\ge0.
$$
Adding these two inequalities becomes
\begin{align}
\label{5-6}
0&\le |\eta_1|^{p-2}\mathrm{tr}(A_p(\eta_1)(X+L_2I))-|\eta_2|^{p-2}\mathrm{tr}(A_p(\eta_2)(Y-L_2I)) \nonumber\\
&\quad+a(\overline{x})|\eta_1|^{q-2}\mathrm{tr}(A_q(\eta_1)(X+L_2I))-a(\overline{y})|\eta_2|^{q-2}\mathrm{tr}(A_q(\eta_2)(Y-L_2I)) \nonumber\\
&\quad+|\eta_1|^{q-2}\eta_1\cdot Da(\overline{x})-|\eta_2|^{q-2}\eta_2\cdot Da(\overline{y})+f(\overline{x},u(\overline{x}),\eta_1)-f(\overline{y},u(\overline{y}),\eta_2)  \nonumber\\
&=:I_1+I_2+I_3+I_4.
\end{align}
In what follows, our goal is to justify $I_1+I_2+I_3+I_4<0$ under suitable conditions, which reaches a contradiction so that the claim \eqref{5-1} is precisely true. First, we examine the term $I_2$ as
\begin{align*}
I_2&=(a(\overline{x})-a(\overline{y}))|\eta_1|^{q-2}\mathrm{tr}(A_q(\eta_1)(X+L_2I))\\
&\quad+a(\overline{y})(|\eta_1|^{q-2}-|\eta_2|^{q-2})\mathrm{tr}(A_q(\eta_1)(X+L_2I))\\
&\quad+a(\overline{y})|\eta_2|^{q-2}[\mathrm{tr}(A_q(\eta_1)(X+L_2I))-\mathrm{tr}(A_q(\eta_2)(Y-L_2I))]\\
&=(a(\overline{x})-a(\overline{y}))|\eta_1|^{q-2}\mathrm{tr}(A_q(\eta_1)(X+L_2I))+a(\overline{y})(|\eta_1|^{q-2}-|\eta_2|^{q-2})\mathrm{tr}(A_q(\eta_1)(X+L_2I))\\
&\quad+a(\overline{y})|\eta_2|^{q-2}\mathrm{tr}(A_q(\eta_1)(X-Y))+a(\overline{y})|\eta_2|^{q-2}\mathrm{tr}((A_q(\eta_1)-A_q(\eta_2))Y)\\
&\quad+L_2a(\overline{y})|\eta_2|^{q-2}(\mathrm{tr}(A_q(\eta_1)+\mathrm{tr}(A_q(\eta_2)))\\
&=:I_{21}+I_{22}+I_{23}+I_{24}+I_{25}.
\end{align*}
Let us mention that the treatment of $I_2$ is similar to that of $J_1$ in \cite[Lemma 6.1]{FZ21}. However, we here give some details for the sake of readability. For $I_{21}$, by virtue of \eqref{5-2}, \eqref{5-5} and the eigenvalues of $A_q(\eta_1)$, we have
\begin{align*}
I_{21}&\le C(\beta,q,\|a\|_{C^1(B_1)})|\overline{x}-\overline{y}| L_1^{q-2}|\overline{x}-\overline{y}|^{(\beta-1)(q-2)}(n\|A_q(\eta_1)\|\|X\|+nL_2\|A_q(\eta_1)\|)\\
&\le C(n,\beta,q,\|a\|_{C^1(B_1)})L_1^{q-2}|\overline{x}-\overline{y}|^{(\beta-1)(q-2)+1}(L_1|\overline{x}-\overline{y}|^{\beta-2}+L_2).
\end{align*}
We now consider $I_{22}$,
\begin{align*}
I_{22}&\le C(q,\beta)a(\overline{y})L_1^{q-3}|\overline{x}-\overline{y}|^{(\beta-1)(q-3)}L_2(n\|A_q(\eta_1)\|\|X\|+nL_2\|A_q(\eta_1)\|)\\
&\le C(n,\beta,q)a(\overline{y})L_2L_1^{q-3}|\overline{x}-\overline{y}|^{(\beta-1)(q-3)}(L_1|\overline{x}-\overline{y}|^{\beta-2}+L_2).
\end{align*}
Indeed, we evaluate $|\eta_1|^{q-2}-|\eta_2|^{q-2}$ as
\begin{align*}
|\eta_1|^{q-2}-|\eta_2|^{q-2}&=(q-2)|\xi_{12}|^{q-3}(|\eta_1|-|\eta_2|)\\
&\le |q-2|C(q,\beta)L_1^{q-3}|\overline{x}-\overline{y}|^{(\beta-1)(q-3)}|\eta_1-\eta_2|\\
&\le C(q,\beta)L_2L^{q-3}_1|\overline{x}-\overline{y}|^{(\beta-1)(q-3)},
\end{align*}
where $|\xi_{12}|$ is between $|\eta_1|$ and $|\eta_2|$, and we have used \eqref{5-2} and $|\eta_1-\eta_2|\le 4L_2$. Due to \eqref{5-3} and \eqref{5-4}, any eigenvalue of $X-Y$ is nonnegative and at least one eigenvalue, denoted by $\overline{\lambda}(X-Y)$, is less than or equal to $8\beta\frac{\beta-1}{3-\beta}L_1|\overline{x}-\overline{y}|^{\beta-2}$. Next, we deal with the term $I_{23}$ as
\begin{align*}
I_{23}&\leq a(\overline{y})|\eta_2|^{q-2}\sum^n_{i=1}\lambda_i(A_q(\eta_1))\lambda_i(X-Y)\\
&\leq a(\overline{y})|\eta_2|^{q-2}\min\{1,q-1\}\overline{\lambda}(X-Y)\\
&\leq a(\overline{y})C(q,\beta)L_1^{q-2}|\overline{x}-\overline{y}|^{(\beta-1)(q-2)}\left(-8\beta\frac{1-\beta}{3-\beta}\right)L_1|\overline{x}-\overline{y}|^{\beta-2}\\
&\le -C(q,\beta)a(\overline{y})L^{q-1}_1|\overline{x}-\overline{y}|^{(\beta-1)(q-1)-1}
\end{align*}
with $\lambda_1(A_q(\eta_1))\le \lambda_2(A_q(\eta_1))\le \cdots\le \lambda_n(A_q(\eta_1))$ and $\lambda_1(X-Y)\le\lambda_2(X-Y)\le\cdots\le \lambda_n(X-Y)$. Here we need to note $0<\beta<1$. In order to evaluate $I_{24}$, we first notice
\begin{align*}
\|A_q(\eta_1)-A_q(\eta_2)\|
&\le 2|q-2|\left|\frac{\eta_1}{|\eta_1|}-\frac{\eta_2}{|\eta_2|}\right|\\
&\leq2|q-2|\max\left\{\frac{|\eta_1-\eta_2|}{|\eta_1|},\frac{|\eta_1-\eta_2|}{|\eta_2|}\right\}\\
&\leq C(q,\beta)\frac{L_2}{L_1|\overline{x}-\overline{y}|^{\beta-1}}.
\end{align*}
Thereby, from \eqref{5-2}, \eqref{5-5} and the preceding inequality,
\begin{align*}
I_{24}&\leq a(\overline{y})|\eta_2|^{q-2}n\|A_q(\eta_1)-A_q(\eta_2)\|\|Y\|\\
&\leq C(n,q,\beta)a(\overline{y})(L_1|\overline{x}-\overline{y}|^{\beta-1})^{q-2}\frac{L_2}{L_1|\overline{x}-\overline{y}|^{\beta-1}} L_1|\overline{x}-\overline{y}|^{\beta-2}\\
&=C(n,q,\beta)a(\overline{y})L_2L_1^{q-2}|\overline{x}-\overline{y}|^{(\beta-1)(q-2)-1}.
\end{align*}
The term $I_{25}$ finally can be treated by
\begin{align*}
I_{25}\le a(\overline{y})|\eta_2|^{q-2}L_22n\max\{1,q-1\}\le C(n,q,\beta)a(\overline{y})L_2L_1^{q-2}|\overline{x}-\overline{y}|^{(\beta-1)(q-2)}.
\end{align*}
Now merging these estimates of $I_{21}$--$I_{25}$ yields that
\begin{align*}
I_{2}\le& -C(q,\beta)a(\overline{y})L_1^{q-1}|\overline{x}-\overline{y}|^{(\beta-1)(q-1)-1}+C(n,q,\beta)a(\overline{y})L_2L_1^{q-2}|\overline{x}-\overline{y}|^{(\beta-1)(q-2)-1}\\
&+C(n,q,\beta)a(\overline{y})L_2^2L_1^{q-3}|\overline{x}-\overline{y}|^{(\beta-1)(q-3)}+C(n,q,\beta,\|a\|_{C^1(B_1)})L_1^{q-1}|\overline{x}-\overline{y}|^{(\beta-1)(q-1)}\\
&+C(n,q,\beta,\|a\|_{C^1(B_1)})L_2L_1^{q-2}|\overline{x}-\overline{y}|^{(\beta-1)(q-2)+1},
\end{align*}
where we have employed the fact that $|\overline{x}-\overline{y}|<1$. Analogously, we could derive
\begin{align*}
I_{1}\le& -C(p,\beta)L_1^{p-1}|\overline{x}-\overline{y}|^{(\beta-1)(p-1)-1}+C(n,p,\beta)L_2L_1^{p-2}|\overline{x}-\overline{y}|^{(\beta-1)(p-2)-1}\\
&+C(n,p,\beta)L_2^2L_1^{q-3}|\overline{x}-\overline{y}|^{(\beta-1)(p-3)},
\end{align*}
where we just note $a(\cdot)\equiv1$. The term $I_3$ is directly estimated as
$$
I_3\le |\eta_1|^{q-1}|Da(\overline{x})|+|\eta_2|^{q-1}|Da(\overline{y})|\le C(q,\beta,\|a\|_{C^1(B_1)})L^{q-1}_1|\overline{x}-\overline{y}|^{(\beta-1)(q-1)}
$$
by applying \eqref{5-2}. As for $I_4$, according to the growth condition on $f$,
\begin{align*}
I_4&\le |f(\overline{x},u(\overline{x}),\eta_1)|+|f(\overline{y},u(\overline{y}),\eta_2)|\\
&\le \gamma(|u(\overline{x})|)(|\eta_1|^{p-1}+a(\overline{x})|\eta_1|^{q-1})+\Phi(\overline{x})+\gamma(|u(\overline{y})|)(|\eta_2|^{p-1}+a(\overline{y})|\eta_2|^{q-1})+\Phi(\overline{y})\\
&\le \gamma_\infty C(p,q,\beta)(L_1^{p-1}|\overline{x}-\overline{y}|^{(\beta-1)(p-1)}+\|a\|_{L^\infty(B_1)}L_1^{q-1}|\overline{x}-\overline{y}|^{(\beta-1)(q-1)})+\|\Phi\|_{L^\infty(B_1)}
\end{align*}
with $\gamma_\infty:=\max_{t\in[0,\|u\|_{L^\infty(B_1)}]} \gamma(t)$.

We eventually gather the estimates on $I_1$--$I_4$ with \eqref{5-6} to infer that
\begin{align*}
0&\leq\big[-C(p,\beta)L^{p-1}_1|\overline{x}-\overline{y}|^{(\beta-1)(p-1)-1}+C(n,p,\beta)L_2L_1^{p-2}|\overline{x}-\overline{y}|^{(\beta-1)(p-2)-1}\\
&\quad+C(n,p,\beta)L_2^2L_1^{p-3}|\overline{x}-\overline{y}|^{(\beta-1)(p-3)}+C(p,q,\beta,\gamma_\infty)L_1^{p-1}|\overline{x}-\overline{y}|^{(\beta-1)(p-1)}\\
&\quad+C(n,p,q,\beta,\|a\|_{C^1(B_1)},\gamma_\infty)L_1^{q-1}|\overline{x}-\overline{y}|^{(\beta-1)(q-1)}+\|\Phi\|_{L^\infty(B_1)}\\
&\quad+C(n,q,\beta,\|a\|_{C^1(B_1)})L_2L_1^{q-2}|\overline{x}-\overline{y}|^{(\beta-1)(q-2)+1}\big]\\
&\quad+a(\overline{y})\big[-C(q,\beta)L^{q-1}_1|\overline{x}-\overline{y}|^{(\beta-1)(q-1)-1}
+C(n,q,\beta)L_2L_1^{q-2}|\overline{x}-\overline{y}|^{(\beta-1)(q-2)-1}\\
&\quad+C(n,q,\beta)L^2_2L_1^{q-3}|\overline{x}-\overline{y}|^{(\beta-1)(q-3)}\big].
\end{align*}
Now our aim is to make the term at the right-hand side of the above display to be strictly less than 0, through choosing $L_1$ large enough. We first select $L_1$ so large that
\begin{equation*}
\begin{cases}
\frac{1}{3}C(q,\beta)L_1^{q-1}|\overline{x}-\overline{y}|^{(\beta-1)(q-1)-1}\geq C(n,q,\beta)L_2L^{q-2}_1|\overline{x}-\overline{y}|^{(\beta-1)(q-2)-1},\\[2mm]
\frac{1}{3}C(q,\beta)L_1^{q-1}|\overline{x}-\overline{y}|^{(\beta-1)(q-1)-1}\geq C(n,q,\beta)L_2^2L^{q-3}_1|\overline{x}-\overline{y}|^{(\beta-1)(q-3)},
\end{cases}
\end{equation*}
i.e.,
\begin{equation*}
L_1|\overline{x}-\overline{y}|^{\beta-1}\geq C(n,q,\beta,L_2).
\end{equation*}
This can be realized if $L_1$ is sufficiently large, since $|\overline{x}-\overline{y}|<1$ and the power of it is negative. Next, we proceed to choose $L_1$ so large that
\begin{equation*}
\begin{cases}
\frac{1}{7}C(p,\beta)L_1^{p-1}|\overline{x}-\overline{y}|^{(\beta-1)(p-1)-1}\geq C(n,p,\beta)L_2L^{p-2}_1|\overline{x}-\overline{y}|^{(\beta-1)(p-2)-1},\\[2mm]
\frac{1}{7}C(p,\beta)L_1^{p-1}|\overline{x}-\overline{y}|^{(\beta-1)(p-1)-1}\geq C(n,p,\beta)L_2^2L^{p-3}_1|\overline{x}-\overline{y}|^{(\beta-1)(p-3)},\\[2mm]
\frac{1}{7}C(p,\beta)L_1^{p-1}|\overline{x}-\overline{y}|^{(\beta-1)(p-1)-1}\geq C(p,q,\beta,\gamma_\infty)L_1^{p-1}|\overline{x}-\overline{y}|^{(\beta-1)(p-1)},\\[2mm]
\frac{1}{7}C(p,\beta)L_1^{p-1}|\overline{x}-\overline{y}|^{(\beta-1)(p-1)-1}\geq C(n,p,q,\beta,\|a\|_{C^1(B_1)},\gamma_\infty)L_1^{q-1}|\overline{x}-\overline{y}|^{(\beta-1)(q-1)},\\[2mm]
\frac{1}{7}C(p,\beta)L_1^{p-1}|\overline{x}-\overline{y}|^{(\beta-1)(p-1)-1}\geq C(n,q,\beta,\|a\|_{C^1(B_1)})L_2L_1^{q-2}|\overline{x}-\overline{y}|^{(\beta-1)(q-2)+1},\\[2mm]
\frac{1}{7}C(p,\beta)L_1^{p-1}|\overline{x}-\overline{y}|^{(\beta-1)(p-1)-1}\geq \|\Phi\|_{L^\infty(B_1)}.
\end{cases}
\end{equation*}
We arrange the previous display as
\begin{equation}
\label{5-7}
\begin{cases}
L_1|\overline{x}-\overline{y}|^{\beta-1}\geq C(n,p,\beta,L_2,\|\Phi\|_{L^\infty(B_1)}),\\[2mm]
|\overline{x}-\overline{y}|^{-1}\geq C(p,q,\beta,\gamma_\infty),\\[2mm]
L_1^{p-q}|\overline{x}-\overline{y}|^{(\beta-1)(p-q)-1}\geq C(n,p,q,\beta,\|a\|_{C^1(B_1)},\gamma_\infty),\\[2mm]
L_1^{p-q+1}|\overline{x}-\overline{y}|^{(\beta-1)(p-q+1)-2}\geq C(n,p,q,\beta,\|a\|_{C^1(B_1)},L_2).
\end{cases}
\end{equation}
Making use of the fact $|\overline{x}-\overline{y}|\le \left(\frac{2\|u\|_{L^\infty(B_1)}}{L_1}\right)^\frac{1}{\beta}$, we can pick such large $L_1>1$ that the first two inequalities of \eqref{5-7} hold true. To assure the inequality $\eqref{5-7}_3$ holds, we first require
$$
 (\beta-1)(p-q)-1<0 \ \Rightarrow \ q<p+\frac{1}{1-\beta}\quad (\beta\in(0,1)).
 $$
We in turn enforce
\begin{align*}
L_1^{p-q}|\overline{x}-\overline{y}|^{(\beta-1)(p-q)-1}&\ge L_1^{p-q}\left(\frac{2\|u\|_{L^\infty(B_1)}}{L_1}\right)^\frac{(\beta-1)(p-q)-1}{\beta}\\
&\geq C(n,p,q,\beta,\|a\|_{C^1(B_1)},\gamma_\infty),
\end{align*}
that is,
$$
L_1^\frac{p-q+1}{\beta}\ge C(n,p,q,\beta,\|a\|_{C^1(B_1)},\gamma_\infty,\|u\|_{L^\infty(B_1)})
$$
which is true precisely if
$$
p-q+1>0\ \Rightarrow \ q<p+1.
$$
Under this condition, the inequality $\eqref{5-7}_4$ shall hold true by choosing $L_1$ large.

As has been shown above, if $q<p+1$, we can select $L_1$ large enough, which depends on $n,p,q,\beta,\|a\|_{C^1(B_1)},\gamma_\infty$, $\|u\|_{L^\infty(B_1)}$ and $\|\Phi\|_{L^\infty(B_1)}$, to get
$$
0\le -\frac{1}{7}C(p,\beta)L_1^{p-1}|\overline{x}-\overline{y}|^{(\beta-1)(p-1)-1}-\frac{1}{3}C(q,\beta)a(\overline{y})L_1^{q-1}|\overline{x}-\overline{y}|^{(\beta-1)(q-1)-1}<0.
$$
This is a contradiction. Thus the claim \eqref{5-1} holds true, which means that the viscosity solution $u$ is locally $\beta$-H\"{o}lder continuous. The proof is finished now.
\end{proof}

Based on the local $\beta$-H\"{o}lder continuity of $u$ in Lemma \ref{lem5-1}, we could further deduce $u$ is locally Lipschitz continuous.

\begin{lemma}[Local Lipschitz continuity]
\label{lem5-2}
Let $u$ be a bounded viscosity solution to \eqref{main} in $B_1$. Under the assumptions that $0\le a(x)\in C^1(B_1)$, $p\le q\le p+\frac{1}{2}$ and \eqref{0-2}, there is a constant $C$ that depends on $n,p,q,\gamma_\infty,\|a\|_{C^1(B_1)}$, $\|u\|_{L^\infty(B_1)}$ and $\|\Phi\|_{L^\infty(B_1)}$, such that
$$
|u(x)-u(y)|\le C|x-y|
$$
for all $x,y\in B_{1/2}$. Here $\gamma_\infty:=\max_{t\in[0,\|u\|_{L^\infty(B_1)}]} \gamma(t)$.
\end{lemma}

\begin{proof}
Let $x_0,y_0\in B_{1/2}$. Construct an auxiliary function
$$
\Psi(x,y):=u(x)-u(y)-M_1\varphi(|x-y|)-\frac{M_2}{2}|x-x_0|^2-\frac{M_2}{2}|y-y_0|^2, \quad M_1,M_2>0,
$$
where
\begin{equation*}
\varphi(t):=\begin{cases}t-\kappa_0t^\nu,  &\textmd{if } {0\leq t\leq t_1:=\left(\frac{1}{4\nu\kappa_0}\right)^\frac{1}{\nu-1}},\\[2mm]
\varphi(t_1),  &\textmd{if } {t>t_1}
\end{cases}
\end{equation*}
with $1<\nu<2$ and $0<\kappa_0<1$ such that $2<t_1$. We are ready to verify $\Psi(x,y)\le 0$ for $(x,y)\in B_{3/4}\times B_{3/4}$ under the appropriate choice of $M_1,M_2$, which leads to Lipschitz continuity of $u$. We argue by contradiction. Suppose that $\Psi$ reaches its positive maximum at $(\hat{x},\hat{y})\in \overline{B}_{3/4}\times \overline{B}_{3/4}$. As in the beginning of the proof of Lemma \ref{lem5-1}, we can see that
$\hat{x}\neq\hat{y}$ and $\hat{x},\hat{y}\in B_{3/4}$ for large $M_2\geq C\|u\|_{L^\infty(B_1)}$. Moreover, by Lemma \ref{lem5-1}, we know that
$$
|u(x)-u(y)|\leq c_\beta|x-y|^\beta \quad \text{for } x,y\in B_{3/4},
$$
where $c_\beta$ is the same as the $C$ in Lemma \ref{lem5-1}. From the assumptions, we further get
\begin{equation}
\label{5-8}
M_2|\hat{x}-x_0|,M_2|\hat{y}-y_0|\le c_\beta|\hat{x}-\hat{y}|^\frac{\beta}{2}
\end{equation}
by adjusting the constants (by letting $2M_2\le c_\beta$). Additionally, it follows, by selecting $\kappa_0$ sufficiently small, that
$$
M_1(|\hat{x}-\hat{y}|-\kappa_0|\hat{x}-\hat{y}|^\nu)\le 2\|u\|_{L^\infty(B_1)},
$$
i.e.,
\begin{equation}
\label{5-8-1}
|\hat{x}-\hat{y}|\le \frac{4\|u\|_{L^\infty(B_1)}}{M_1}.
\end{equation}
Indeed, due to $|\hat{x}-\hat{y}|\le2$ and $\nu-1>0$, we can fix $\kappa_0\in(0,1)$ such that $\frac{1}{2}\le 1-\kappa_0|\hat{x}-\hat{y}|^{\nu-1}$.

From theorem of sums, for any $\mu>0$, there exist $X,Y\in\mathcal{S}^n$ such that
$$
(\eta_1,X+M_2I)\in \overline{J}^{2,+}u(\hat{x}) \quad \text{and}\quad (\eta_2,Y-M_2I)\in \overline{J}^{2,-}u(\hat{y})
$$
and
\begin{equation}
\label{5-9}
-(\mu+2\|B\|)\left(\begin{array}{cc}
I & \\
 &I
\end{array}
\right)\leq\left(\begin{array}{cc}
X&\\
&-Y
\end{array}
\right)\leq
\left(\begin{array}{cc}
B&-B\\
-B&B
\end{array}
\right)+\frac{2}{\mu}\left(\begin{array}{cc}
B^2&-B^2\\
-B^2&B^2
\end{array}
\right),
\end{equation}
where
\begin{align*}
&\eta_1=M_1\varphi'(|\hat{x}-\hat{y}|)\frac{\hat{x}-\hat{y}}{|\hat{x}-\hat{y}|}+M_2(\hat{x}-x_0),\\
&\eta_2=M_1\varphi'(|\hat{x}-\hat{y}|)\frac{\hat{x}-\hat{y}}{|\hat{x}-\hat{y}|}-M_2(\hat{y}-y_0)
\end{align*}
and
\begin{align*}
B&=M_1\varphi''(|\hat{x}-\hat{y}|)\frac{\hat{x}-\hat{y}}{|\hat{x}-\hat{y}|}
\otimes\frac{\hat{x}-\hat{y}}{|\hat{x}-\hat{y}|}+\frac{M_1\varphi'(|\hat{x}-\hat{y}|)}{|\hat{x}-\hat{y}|}
\left(I-\frac{\hat{x}-\hat{y}}{|\hat{x}-\hat{y}|}
\otimes\frac{\hat{x}-\hat{y}}{|\hat{x}-\hat{y}|}\right).
\end{align*}
Notice that, for $t\in[0,t_1]$,
\begin{equation}
\label{5-10}
\begin{cases}
\varphi'(t)=1-\nu\kappa_0t^{\nu-1}, \\[2mm]
\varphi''(t)=-\nu(\nu-1)\kappa_0t^{\nu-2},
\end{cases}
\end{equation}
and then $\frac{3}{4}\leq \varphi'(t)\leq1$ and $\varphi''(t)<0$ when $t\in(0,2]$. Through straightforward calculation we get
\begin{equation}
\label{5-11}
\frac{M_1}{2}\leq|\eta_1|,|\eta_2|\leq2M_1, \quad \text{if } M_1\geq4c_\beta,
\end{equation}
\begin{equation}
\label{5-12}
\|B\|\leq M_1\frac{\varphi'(|\hat{x}-\hat{y}|)}{|\hat{x}-\hat{y}|}
\end{equation}
and
\begin{equation}
\label{5-13}
\|B^2\|\leq M_1^2\left(|\varphi''(|\hat{x}-\hat{y}|)|+\frac{\varphi'(|\hat{x}-\hat{y}|)}{|\hat{x}-\hat{y}|}\right)^2.
\end{equation}
According to \eqref{5-9}, we obtain $X\le Y$ and furthermore $\|X\|,\|Y\|\le 2\|B\|+\mu$. Via taking
$$
\mu=4M_1\left(|\varphi''(|\hat{x}-\hat{y}|)|+\frac{\varphi'(|\hat{x}-\hat{y}|)}{|\hat{x}-\hat{y}|}\right),
$$
we have, for $\xi=\frac{\hat{x}-\hat{y}}{|\hat{x}-\hat{y}|}$,
\begin{equation}
\label{5-14}
\langle(X-Y)\xi,\xi\rangle\leq 4\left(\langle B\xi,\xi\rangle+\frac{2}{\mu}\langle B^2\xi,\xi\rangle\right)\leq 2M_1\varphi''(|\hat{x}-\hat{y}|)<0.
\end{equation}
It follows from the last inequality that at lowest one eigenvalue of $X-Y$, denoted by $\hat{\lambda}$, is smaller than $2M_1\varphi''(|\hat{x}-\hat{y}|)<0$. Besides, putting together \eqref{5-12}, \eqref{5-13} and \eqref{5-9}, we derive
\begin{align}
\label{5-15}
\|Y\| &\leq 2|\langle B\overline{\xi},\overline{\xi}\rangle|+\frac{4}{\mu}|\langle B^2\overline{\xi},\overline{\xi}\rangle|  \nonumber\\
&\leq4M_1\left(|\varphi''(|\hat{x}-\hat{y}|)|+\frac{\varphi'(|\hat{x}-\hat{y}|)}{|\hat{x}-\hat{y}|}\right),
\end{align}
where $\overline{\xi}$ is a unit vector.

Because $u$ is a viscosity solution, we arrive at
\begin{align}
\label{5-16}
0&\le |\eta_1|^{p-2}\mathrm{tr}(A_p(\eta_1)(X+M_2I))-|\eta_2|^{p-2}\mathrm{tr}(A_p(\eta_2)(Y-M_2I)) \nonumber\\
&\quad+a(\hat{x})|\eta_1|^{q-2}\mathrm{tr}(A_q(\eta_1)(X+M_2I))-a(\hat{y})|\eta_2|^{q-2}\mathrm{tr}(A_q(\eta_2)(Y-M_2I)) \nonumber\\
&\quad+|\eta_1|^{q-2}\eta_1\cdot Da(\hat{x})-|\eta_2|^{q-2}\eta_2\cdot Da(\hat{y})+f(\hat{x},u(\hat{x}),\eta_1)-f(\hat{y},u(\hat{y}),\eta_2)  \nonumber\\
&=:T_1+T_2+T_3+T_4.
\end{align}
In the following, we want to prove $T_1+T_2+T_3+T_4<0$ by taking $M_1$ large enough, the procedure of which is analogous to that in proof of Lemma \ref{lem5-1}. Here we shall briefly write down it. We first consider the term $T_2$,
\begin{align*}
T_2
&=(a(\hat{x})-a(\hat{y}))|\eta_2|^{q-2}\mathrm{tr}(A_q(\eta_2)(Y-M_2I))+a(\hat{x})|\eta_1|^{q-2}\mathrm{tr}(A_q(\eta_1)(X-Y))\\
&\quad+a(\hat{x})(|\eta_1|^{q-2}-|\eta_2|^{q-2})\mathrm{tr}(A_q(\eta_2)Y)+a(\hat{x})|\eta_1|^{q-2}\mathrm{tr}((A_q(\eta_1)-A_q(\eta_2))Y)\\
&\quad+M_2a(\hat{x})[|\eta_1|^{q-2}\mathrm{tr}(A_q(\eta_1))+|\eta_2|^{q-2}\mathrm{tr}(A_q(\eta_2))]\\
&=:T_{21}+T_{22}+T_{23}+T_{24}+T_{25}.
\end{align*}
In view of \eqref{5-10}, \eqref{5-11} and \eqref{5-15}, there holds that
\begin{align*}
T_{21}&\le C(q,\|a\|_{C^1(B_1)})M_1^{q-2}|\hat{x}-\hat{y}|(n\|A_q(\eta_2)\|\|Y\|+nqM_2)\\
&\le C(n,q,\|a\|_{C^1(B_1)})M_1^{q-2}[M_1(1+|\hat{x}-\hat{y}||\varphi''(|\hat{x}-\hat{y}|)|)+M_2]\\
&\le C(n,q,\|a\|_{C^1(B_1)})M_1^{q-2}[M_1(1+|\hat{x}-\hat{y}|^{\nu-1})+M_2]\\
&\le C(n,q,\|a\|_{C^1(B_1)})(M_1^{q-1}+M_1^{q-2}M_2).
\end{align*}
Applying the mean value theorem along with \eqref{5-8}, \eqref{5-10} and \eqref{5-11}, we obtain
$$
\big||\eta_1|^{q-2}-|\eta_2|^{q-2}\big|\le C(q,c_\beta)M_1^{q-3}|\hat{x}-\hat{y}|^\frac{\beta}{2},
$$
which indicates that
\begin{align*}
T_{23}&\le C(n,q,c_\beta)a(\hat{x})M_1^{q-3}|\hat{x}-\hat{y}|^\frac{\beta}{2}\|Y\|\\
&\le C(n,c_\beta,q)a(\hat{x})M_1^{q-2}(|\hat{x}-\hat{y}|^{\frac{\beta}{2}-1}+|\hat{x}-\hat{y}|^{\nu-2}).
\end{align*}
Thanks to \eqref{5-11} and \eqref{5-14}, we derive
\begin{align*}
T_{22}&\leq a(\hat{x})|\eta_1|^{q-2}\sum^n_{i=1}\lambda_i(A_q(\eta_1))\lambda_i(X-Y)\\
&\leq C(q)a(\hat{x})M_1^{q-1}\varphi''(|\hat{x}-\hat{y}|)\\
&= -C(q,\nu,\kappa_0)a(\hat{x})M_1^{q-1}|\hat{x}-\hat{y}|^{\nu-2}.
\end{align*}
According to \eqref{5-8}, \eqref{5-11} and \eqref{5-15},
\begin{align*}
T_{24}&\leq a(\hat{x})|\eta_1|^{q-2}n\|A_q(\eta_1)-A_q(\eta_2)\|\|Y\|\\
&\leq C(n,q)a(\hat{x})|\eta_1|^{q-2}\max\left\{\frac{|\eta_1-\eta_2|}{|\eta_1|},\frac{|\eta_1-\eta_2|}{|\eta_2|}\right\}\|Y\|\\
&\le C(n,q,c_\beta)a(\hat{x})M_1^{q-2}|\hat{x}-\hat{y}|^\frac{\beta}{2}\left(|\varphi''(|\hat{x}-\hat{y}|)|+\frac{\varphi'(|\hat{x}-\hat{y}|)}{|\hat{x}-\hat{y}|}\right)\\
&\le C(n,q,c_\beta)a(\hat{x})M_1^{q-2}(|\hat{x}-\hat{y}|^{\frac{\beta}{2}-1}+|\hat{x}-\hat{y}|^{\nu-2}).
\end{align*}
Finally, by \eqref{5-11} we have
\begin{align*}
T_{25}\le C(n,q)a(\hat{x})M_2M_1^{q-2}.
\end{align*}
Thus taking $\nu=\frac{\beta}{2}+1$ and combining these previous inequalities yields that
\begin{align*}
T_{2}\le& -C(q,\beta)a(\hat{x})M_1^{q-1}|\hat{x}-\hat{y}|^{\frac{\beta}{2}-1}+C(n,q,c_\beta)a(\hat{x})M_1^{q-2}|\hat{x}-\hat{y}|^{\frac{\beta}{2}-1}\\
&+C(n,q,\|a\|_{C^1(B_1)})(M_1^{q-1}+M_1^{q-2}M_2).
\end{align*}
Similarly, we have
$$
T_1\le -C(p,\beta)M_1^{p-1}|\hat{x}-\hat{y}|^{\frac{\beta}{2}-1}+C(n,p,c_\beta)M_1^{p-2}|\hat{x}-\hat{y}|^{\frac{\beta}{2}-1}
+C(n,p)M_1^{p-2}M_2.
$$
For $T_3$, it is easy to get
$$
T_3\le C(q,\|a\|_{C^1(B_1)})M_1^{q-1}.
$$
Owing to the growth condition on $f$, we can see that
\begin{align*}
T_4&\le \gamma_\infty(|\eta_1|^{p-1}+a(\hat{x})|\eta_1|^{q-1})+\gamma_\infty(|\eta_2|^{p-1}+a(\hat{y})|\eta_2|^{q-1})+2\|\Phi\|_{L^\infty(B_1)}\\
&\le C(p,q,\gamma_\infty,\|a\|_{L^\infty(B_1)})(M_1^{p-1}+M_1^{q-1})+2\|\Phi\|_{L^\infty(B_1)}\\
&\le C(p,q,\gamma_\infty,\|a\|_{L^\infty(B_1)},\|\Phi\|_{L^\infty(B_1)})M_1^{q-1}
\end{align*}
with $\gamma_\infty:=\max_{t\in[0,\|u\|_{L^\infty(B_1)}]} \gamma(t)$. Here we note $M_1>1$ is a sufficiently large number. It follows from merging the estimates on $T_1$--$T_4$ with \eqref{5-16} that
\begin{align}
\label{5-17}
0&\leq\big[-C(p,\beta)M^{p-1}_1|\hat{x}-\hat{y}|^{\frac{\beta}{2}-1}+C(n,p,c_\beta)M^{p-2}_1|\hat{x}-\hat{y}|^{\frac{\beta}{2}-1} \nonumber\\
&\quad+C(p,q,\gamma_\infty,\|a\|_{C^1(B_1)},\|\Phi\|_{L^\infty(B_1)})M_1^{q-1}\big] \nonumber\\
&\quad+a(\hat{x})\big[-C(q,\beta)M^{q-1}_1|\hat{x}-\hat{y}|^{\frac{\beta}{2}-1}+C(n,q,c_\beta)M_1^{q-2}|\hat{x}-\hat{y}|^{\frac{\beta}{2}-1}\big],
\end{align}
where we have used the relation $M_1\ge M_2$ to simplify the display. To get a contradiction, we have to select such large $M_1$ that
\begin{equation}
\label{5-18}
\begin{cases}
\frac{1}{3}C(p,\beta)M^{p-1}_1|\hat{x}-\hat{y}|^{\frac{\beta}{2}-1}\geq C(n,p,c_\beta)M_1^{p-2}|\hat{x}-\hat{y}|^{\frac{\beta}{2}-1},\\[2mm]
\frac{1}{3}C(p,\beta)M^{p-1}_1|\hat{x}-\hat{y}|^{\frac{\beta}{2}-1}\geq C(p,q,\gamma_\infty,\|a\|_{C^1(B_1)},\|\Phi\|_{L^\infty(B_1)})M_1^{q-1},\\[2mm]
\frac{1}{2}C(q,\beta)M^{q-1}_1|\hat{x}-\hat{y}|^{\frac{\beta}{2}-1}\geq C(n,q,c_\beta)M^{q-2}_1|\hat{x}-\hat{y}|^{\frac{\beta}{2}-1},
\end{cases}
\end{equation}
that is,
\begin{equation*}
\begin{cases}
M_1\geq C(n,p,q,c_\beta),\\[2mm]
M^{p-q}_1|\hat{x}-\hat{y}|^{\frac{\beta}{2}-1}\geq C(p,q,\beta,\gamma_\infty,\|a\|_{C^1(B_1)},\|\Phi\|_{L^\infty(B_1)}).
\end{cases}
\end{equation*}
Remembering \eqref{5-8-1}, we arrive at
$$
M^{p-q}_1|\hat{x}-\hat{y}|^{\frac{\beta}{2}-1}\geq (4\|u\|_{L^\infty(B_1)})^{\frac{\beta}{2}-1}M_1^{p-q+1-\frac{\beta}{2}}.
$$
Now enforcing
$$
q<p+1-\frac{\beta}{2},
$$
we could choose such large $M_1$ that
$$
M_1\ge C(n,p,q,\beta,\gamma_\infty,\|a\|_{C^1(B_1)},\|\Phi\|_{L^\infty(B_1)},\|u\|_{L^\infty(B_1)}),
$$
which ensures \eqref{5-18} holds true. Then the display \eqref{5-17} becomes
$$
0\leq -\frac{1}{3}CM^{p-1}_1|\hat{x}-\hat{y}|^{\frac{\beta}{2}-1}-\frac{1}{2}a(\hat{x})CM^{q-1}_1|\hat{x}-\hat{y}|^{\frac{\beta}{2}-1}<0,
$$
which is a contradiction. Let us mention that we can fix $\beta$ to a specific number so that $M_1$ does not depend on $\beta$. Up to now, we have justified the local Lipschitz continuity of $u$.
\end{proof}

Once Lemma  \ref{lem5-2} is proved in $B_1$, then the case of a  bounded domain $\Omega$ follows by the covering arguments.  Therefore, we finish the proof of Theorem \ref{thm1-1}.

\section*{Acknowledgments}

This work was supported by the National Natural Science Foundation of China (No. 12071098) and the National Postdoctoral Program for Innovative Talents of China (No. BX20220381). The second author was supported by a grant of the Romanian Ministry of Research, Innovation, and Digitization, CNCS/CCCDI-UEFISCDI (project no. PCE 137/2021), within PNCDI III.

\end{document}